\documentclass[11pt]{amsart}
\usepackage{amsmath}
\usepackage{amssymb}
\usepackage{amsthm}
\usepackage{enumerate}
\usepackage{url}
\usepackage{stmaryrd}
\usepackage{color}

\usepackage[vmargin=2cm,hmargin=2cm]{geometry}
\date{}
\newtheorem{thm}{Theorem}[section]
\newtheorem{lem}[thm]{Lemma}

\newtheorem{prop}[thm]{Proposition}
\newtheorem{cor}[thm]{Corollary}

\theoremstyle{definition}
\newtheorem{ex}[thm]{Example}
\newtheorem{rem}[thm]{Remark}
\newtheorem{alg}[thm]{Algorithm}
\numberwithin{equation}{section}
\newcommand{\KK}{{\mathbb K}}
\newcommand{\NN}{{\mathbb N}}

\title{Bases of subalgebras of $\KK\llbracket x\rrbracket $ and $\KK[x]$}
\author{A. Assi} 
\address{Universit\'e d'Angers, Math\'ematiques,
49045 Angers ceded 01, France}
\email{assi@univ-angers.fr} 

\thanks{The first author is partially supported by the project GDR CNRS 2945 and a GENIL-SSV 2014 grant}

\author{P. A. Garc\'{\i}a-S\'anchez}
\address{IEMath-GR and Departamento de \'Algebra, Universidad de Granada, E-18071 Granada, Espa\~na}
\email{pedro@ugr.es} %\url{www.ugr.es/local/pedro}

\thanks{The second author is supported by the projects MTM2014-55367-P, FQM-343,  FQM-5849, and FEDER funds}

\author{V. Micale}
\address{Dipartimento di Matematica e Informatica, Viale A. Doria, 6 - I 95125 – Catania, Italia}
\email{vmicale@dmi.unict.it}

\begin{document}

\newcommand{\card}{\mathrm {card}}

\begin{abstract}
Let $f_1,\ldots, f_s$ be formal power series (respectively polynomials) in the variable $x$. We study the semigroup of orders of the formal series in
the algebra $K\llbracket f_1,\ldots, f_s\rrbracket \subseteq K\llbracket x \rrbracket$ (respectively the semigroup of degrees of polynomials in
$K[f_1,\ldots,f_s]\subseteq K[x]$). We give procedures to compute these semigroups and several applications. We prove in particular that the space curve parametrized by $f_1,\ldots,f_s$ has a flat deformation into a monomial curve.
\end{abstract}

\maketitle

\section{Introduction}

Let $\KK$ be a field and let $\KK\llbracket x\rrbracket$ be the ring of formal power series  over $\KK$. Let $f_1(x),\ldots,f_s(x)$ be $s$ elements of $\KK\llbracket x\rrbracket$  and let $R=\KK\llbracket f_1,\ldots,f_s\rrbracket$  be the subalgebra of $\KK\llbracket x\rrbracket$  generated by $f_1,\ldots,f_s$. Given $f\in R$, let $\mathrm o(f)$ the order of $f$. The set $\mathrm o(R)=\lbrace o(f)\mid  f\in R\rbrace$  is a submonoid of $\NN$, and the knowledge of a system of generators of this monoid is important for the understanding of the subalgebra $R$. When furthermore $\KK\llbracket x\rrbracket$ is an $R$-module  of finite length, then $\mathrm o(R)$ is a numerical semigroup.

A similar construction can be made in the ring of polynomials $\KK[x]$. More precisely let $f_1(x),\ldots,f_s(x)$ be $s$ elements of  $\KK[x]$ and let  $A=\KK[f_1,\ldots,f_s]$ be the subalgebra of  $\KK[x]$ generated by $f_1,\ldots,f_s$. Given $f\in A$, let $\mathrm d(f)$ the  degree of $f$. The set  $\mathrm d(A)= \lbrace d(f)\mid f\in A\rbrace$ is a submonoid   $\mathrm d(A)$ of $\NN$, and the knowledge of a system of generators of this monoid is important for the understanding of the subalgebra  $A$. When furthermore $\KK[x]$ is an $A$-module of finite length, then $\mathrm d(A)$ is a numerical semigroup. 

A numerical semigroup $S$ is a submonoid of the set of nonnegative integers under addition such that  the $\mathbb N\setminus S$ is finite, or equivalently, $\gcd(S)=1$ (the greatest common divisor of the elements of $S$), see for instance \cite{ns}. In this case, there exists a minimum $c\in S$ such that $c+\mathbb N\subseteq S$. We call this element the \emph{conductor} of $S$, and denote it by $\mathrm c(S)$ (the motivation of this name and others coming from Algebraic Geometry is explained in \cite{1,2}).

Assume that $f_i$ is a monomial $x^{a_i}$ for every $i\in\{1,\ldots, s\}$. Then $\mathrm o(R)$ (respectively $\mathrm d(A)$) is generated by $a_1,\ldots,a_s$. In this case, $R\simeq \KK\llbracket X_1,\ldots,X_s\rrbracket /T$ (respectively $A\simeq \KK[X_1,\ldots,X_s]/T$), where $T$ is a prime binomial ideal and, thus,   $\mathrm V(T)$ is a toric variety.

Given a subalgebra $R=\KK[\![f_1,\ldots,f_s]\!]$ (respectively  $A=\KK[f_1,\ldots,f_s]$), the main objective of this paper is to describe an algorithm that calculates a generating system of $\mathrm o(R)$ (respectively $\mathrm d(A)$). The algorithm we present here allows us, by using the technique of homogenization, to construct a flat $\KK\llbracket u\rrbracket$-module (respectively $\KK[u]$-module) which is a deformation of $R$ (respectively $A$) to a binomial ideal. This technique is well known when $R=\KK\llbracket f_1,f_2\rrbracket$ and $\KK$ is algebraically closed field of characteristic zero (see \cite{go} and \cite{tei}). It turns out that the same holds wherever we can associate a semigroup to the local subalgebra, and also that the same technique can be adapted to the global setting. As a particular case we prove that a plane polynomial curve has a deformation into a complete intersection monomial space curve.

The paper is organized as follows. In Section \ref{formal-space} we focus on the local case, namely the case of a subalgebra $R$ of $\KK\llbracket x\rrbracket$. We introduce the notion of basis of $R$ and we show how to construct such a basis. We also show that
if $\mathrm o(R)$ is a numerical semigroup, then every element of a reduced basis is a polynomial. In Section \ref{deformation-local} we show how to construct a deformation from $R$ to a toric ideal (or a formal toric variety) by using the technique of homogenization. In Section \ref{two-local} we focus on the case when $R=\KK[\![f(x),g(x)]\!]$ and $\KK$ is an algebraically closed field of characteristic zero. The existence in this case of the theory of Newton-Puiseux allows us to precise the results of Sections \ref{formal-space} and \ref{deformation-local}. The difference with the procedure presented in Section \ref{formal-space} is that it does not rely in the computation of successive kernels. Then in Sections \ref{polynomial-curve} and \ref{deformation-global} we adapt the local results to the case of a subalgebra $A$ of $\KK[x]$. When $A=\KK[f(x),g(x)]$ and $\KK$ is algebraically closed field of characteristic zero, a basis of $A$ can be obtained by using the theory of approximate roots of the resultant of $X-f(x),Y-g(x)$, which is a polynomial with one place at infinity. 

The procedures presented here have been implemented in \texttt{GAP} (\cite{gap}) and will be part of the forthcoming stable release of the package \texttt{numericalsgps} (\cite{numericalsgps}).

\subsection{Some notation}

We denote by $\langle A\rangle$ the monoid generated by $A$, $A\subseteq \mathbb N$, that is, the set $\{n_{1}x_{1}+\cdots+n_{m}x_{m}\mid m\in \mathbb N, n_{i}\in \mathbb N, x_{i}\in A \hbox{ for all } i\in\{1,\dots,m\}\}$. 

Associated to each numerical semigroup $S$ we can define a natural partial ordering $\le _{S}$, where for two elements $s$ and $r$ in $S$ we have $s\le _{S} r$ if there exists  $u\in S$ such that $r=s+u$. The set $g_{i}$ of minimal elements in $S\setminus \{0\}$ with this ordering is called a {\em minimal set of generators} for $S$. The set of minimal generators is finite since for any $s\in S\setminus\{0\}$, we have $x\not\equiv y\pmod s$ if $x\neq y$ are minimal elements with respect to $\le_{S}$. The cardinality of the minimal generating system is known as the \emph{embedding dimension} of $S$.

\section{Semigroup of a formal space curve}\label{formal-space}

Let $\KK$ be a field. In this section we will consider rings $R$ that are subalgebras of $\KK\llbracket x\rrbracket $ and such that, if we denote  the integral closure of $R$ in its quotient field by $\bar R$, then $\bar R=\KK\llbracket x\rrbracket $ and $\lambda_{R}(\bar{R}/R)<\infty$  where
$\lambda_R(\cdot)$ is the length as $R$-module. Part of the results of this section are inspired in \cite{preprint} and in \cite{m}. An alternative procedure (implemented in Maple) is provided in \cite{cc}. The main difference with our approach is that we do not rely on the multiplicity sequence, and thus we do not need to perform blow-ups. Also we take intrinsic advantage in our implementation of the \texttt{GAP} package \texttt{numericalsgps} (\cite{gap, numericalsgps}).

Let $f=\sum_{i}c_{i}x^{i}\in \bar{R}^*=\bar R\setminus\{0\}$. Define $\mathrm{supp}(f)=\lbrace i, c_i\not=0\rbrace$. We call $\min\mathrm{supp}(f)$ the \emph{order} of $f$ and we denote it by $\mathrm o(f)$. We also set $\mathrm M_\mathrm o(f)=c_{\mathrm o(f)}x^{\mathrm o(f)}$. If $\mathrm M_\mathrm o(f)=x^{\mathrm o(f)}$, then we shall say, by abuse of notation, that $f$ is \emph{monic}. We set $\mathrm o(0)=+\infty$. 

We denote by $\mathrm o(R)$ the set of orders of elements in $R^*=R\setminus\{0\}$, that is, $\mathrm o(R)=\{\mathrm o(f)\mid f\in R^*\}$. We finally set $\mathrm M_\mathrm o(R)=\KK[\mathrm M_\mathrm o(f) \mid f\in R^*]$. 

\begin{prop}\label{o}
  Let $f_{1},f_{2}$ be elements of ${\bar R}^{*}$ and let $a=\min\{\mathrm o(f_{1}),\mathrm o(f_{2})\}$. 

  \begin{enumerate}[(i)]
  \item  $a\le \mathrm o(f_{1}+f_{2})$.
  \item  If $\mathrm o(f_{1})\neq \mathrm o(f_{2})$ then $a=\mathrm o(f_{1}+f_{2})$.
  \item  $\mathrm o(f_{1}f_{2})=\mathrm o(f_{1})+\mathrm o(f_{2})$.
  \end{enumerate}
\end{prop}
\begin{proof}
This follows easily from the definition of order.
\end{proof}

\begin{prop}\label{R1}\cite[Lemma 3, p.486]{4}
 Let $R_{1}$ and $R_{2}$ be rings of our type such that  $R_{1}\subseteq R_{2}$ and $\mathrm o(R_{1})=\mathrm o(R_{2})$. Then $R_{1}=R_{2}$.
\end{prop}

 \begin{prop}\label{Nd}\cite[Proposition 1, p.488]{4}
Let $R$ be a ring of our type. Then $\lambda_{R}(\bar{R}/R)=|\NN\setminus \mathrm o(R)|$.
\end{prop}

The following two results appear in \cite{preprint}, and since this paper is very hard to find, we include the proofs for sake of completeness. 

\begin{prop}\label{o(R)}\cite{preprint}
Let $R$ be a ring of our type. Then $\mathrm o(R)$ is a numerical semigroup.
\end{prop}
\begin{proof}
Since $R$ is a ring and by (iii) of Proposition \ref{o}, we have that $\mathrm o(R)$ is a subsemigroup of $\NN$. By $\lambda_{R}(\bar{R}/R)<\infty$ and Proposition \ref{Nd}, we have the proof.
\end{proof}

\begin{prop}\label{a}\cite{preprint}
Let $R$ be a ring of our type. Then $R$ contains every element $f\in\KK\llbracket x\rrbracket $ of order $\mathrm o(f)\ge c(\mathrm o(R))$.
\end{prop}
\begin{proof}
Use Proposition \ref{R1} with $R_{1}=R$ and $R_{2}=R+x^{\mathrm c(\mathrm o(R))}\KK\llbracket x\rrbracket $.
\end{proof}

This later result allows to work with polynomials instead of series.

Let $f_1,\dots,f_s$ be in $\bar{R}^*$.  Let $R=\KK\llbracket f_1,\dots,f_s\rrbracket $ be a subalgebra of $\KK\llbracket x\rrbracket $ as above, that is, the integral closure of $R$ in its quotient field is $\bar R=\KK\llbracket x\rrbracket $ and $\lambda_{R}(\bar{R}/R)<\infty$.

Under these hypotheses on $R$, we have that $\mathrm o(R)$ is a numerical semigroup (Proposition \ref{o(R)}).

We say that the set  $\{f_1,\dots,f_s\}\subset R^*$ is a \emph{basis} of $R$ if  $\mathrm o(R)=\langle \mathrm o(f_1),\dots,\mathrm o(f_s)\rangle$. The set $\lbrace f_1,\ldots,f_s\rbrace$ is a basis of $R$ if and only if $\mathrm M_\mathrm o(R)=\KK[\mathrm M_\mathrm o(f_1),\ldots,\mathrm M_\mathrm o(f_s)]$.

%In this paper we describe an algorithm to
%construct an order basis from the generators $f_1,\dots,f_n$ and
%we will give some applications.
%:

 \begin{prop}\label{remainder}
 Let the notations be as above. Given $f(x)\in \KK\llbracket x\rrbracket $, there exist $g(x)\in R$ and $r(x)\in \KK\llbracket x\rrbracket $ such that the following conditions hold.
 
 \begin{enumerate}[(i)]
 \item $f(x)=g(x)+r(x)=\sum_{\underline{\alpha}}c_{\underline{\alpha}}f_1^{\alpha_1}×\cdots f_s^{\alpha_s}+r(x)$.
 
 \item If $g(x)\not=0$ (respectively $r(x)\not=0$), then $\mathrm o(g)\geq \mathrm o(f)$ (respectively $\mathrm o(r)\geq \mathrm o(f)$).
 
 \item Either $r(x)=0$ or $\mathrm{supp}(r(x))\subseteq \NN\setminus \langle \mathrm o(f_1),\ldots,\mathrm o(f_s)\rangle$.
 \end{enumerate}
 \end{prop}
 
 \begin{proof} The assertion is clear if $f\in\KK$. Suppose that $f\notin\KK$ and let $f(x)=\sum_{i\geq p}c_ix^i$ with $p=\mathrm o(f)\ge 0$.
 
 \begin{enumerate}
 
 \item If $p\notin \langle \mathrm o(f_1),\ldots,\mathrm o(f_s)\rangle$, then we set $g^1=0, r^1=c_px^p$ and $f^1=f-c_px^p$.
 
 \item If $p\in \langle \mathrm o(f_1),\ldots,\mathrm o(f_s)\rangle$, then  $c_px^p=c_{\underline{\theta}}\mathrm M_0(f_1)^{\theta_1}\cdots  \mathrm M_0(f_s)^{\theta_s}$. We set
 $g^1=c_{\underline{\theta}}f_1^{\theta^1_1}\cdots  f_s^{\theta^1_s}$, $r^1=0$ and
 $f^1=f-g^1$.
 \end{enumerate}
In such a way that $f=f^1+g^1+r^1$, $g^1\in {R}$, either $r^1=0$ or  $\mathrm{supp}(r^1)\subseteq \NN\setminus \langle \mathrm o(f_1),\ldots,\mathrm o(f_s)\rangle$, and if $f^1\not=0$, then $\mathrm o(f^1)>\mathrm o(f)=p$. Then we restart with $f^1$. We construct in this way sequences $(f^k)_{k\geq 1}, (g^k)_{k\geq 1}, (r^k)_{k\geq 1}$ such that for all $k\geq 1, f=f^k+\sum_{i=1}^kg^i+\sum_{i=1}^kr^i$, and
 $\mathrm o(f)<\mathrm o(f^1)<\dots<\mathrm o(f^k)$, $\sum_{i=1}^kg^i\in{R}$, $\mathrm{supp}(\sum_{i=1}^kr^i)\in \NN\setminus \langle \mathrm o(f_1),\ldots,\mathrm o(f_s)\rangle$ and for all $i<j\leq k$, if $g^i\not=0\not= g^j$ (respectively $r^i\not= 0\not= r^j$), then  $\mathrm o(f)\leq \mathrm o(g^i)<\mathrm o(g^j)$ (respectively $\mathrm o(f)\leq \mathrm o(r^i)<\mathrm o(r^j)$). Clearly  $\lim_{k\longrightarrow +\infty}f^k=0$. Hence, if $g=\lim_{k\longrightarrow +\infty}\sum_{i=1}^kg^i$ and $r=\lim_{k\longrightarrow +\infty}\sum_{i=1}^kr^i$, then $f=g+r$ and $g,r$ satisfy the conditions above. 
 \end{proof}
 
 We denote the series $r(x)$ of the proposition above by $\mathrm R_\mathrm o(f,\lbrace f_1,\ldots,f_s\rbrace)$. This series  depend strongly on step (2) of the proof of Proposition 2.6. Let for example $f_1=x^6,f_2=x^4+x^5, f_3=x^2+x^5$, and let $f=x^4$. We have $f=f_2-x^5=f_3^2-f_1f_2-2x^7+x^{11}$. We shall see that $r(x)$ becomes unique if $f_1,\ldots,f_s$ is a basis of $R$ (see Proposition \ref{uniqueness-remainder}). 
%We call the series $r(x)$ of the result above the \emph{remainder} of $f$ with respect to $\lbrace f_1,\ldots,f_s\rbrace$ and 

\begin{prop}\label{carac-local-1}
The set $\lbrace f_1,\ldots,f_s\rbrace$ is a basis of ${R}$ if and only if $\mathrm R_\mathrm o(f,\lbrace f_1,\ldots,f_s\rbrace)=0$ for all $f\in{R}$.
\end{prop}
 \begin{proof} Suppose that $\lbrace f_1,\ldots,f_s\rbrace$ is a basis of ${R}$ and let $f\in{R}$. Let $r(x)=\mathrm R_\mathrm o(f,\lbrace f_1,\ldots,f_s\rbrace)$. Then $r(x)\in {R}$. If $r\not=0$, then  $\mathrm o(r)\in\langle\mathrm o(f_1),\ldots,\mathrm o(f_s)\rangle$, which is a contradiction.
 
Conversely, suppose that  $\lbrace f_1,\ldots,f_s\rbrace$ is not a basis of ${R}$ and let $0\not=f\in {R}$ such that  $\mathrm o(f)\notin\langle \mathrm o(f_1),\ldots,\mathrm o(f_s) \rangle$. We have  $\mathrm R_\mathrm o(f,\lbrace f_1,\ldots,f_s\rbrace)\not=0$, which contradicts the hypothesis.
\end{proof}
 
 \begin{prop}\label{uniqueness-remainder}
 If $\lbrace f_1,\ldots,f_s\rbrace$ is a basis of ${R}$ then for all $f\in \KK\llbracket x\rrbracket$,  $\mathrm R_\mathrm o(f,\lbrace f_1,\ldots,f_s\rbrace)$ is unique.
 \end{prop}
 \begin{proof} Suppose that $f=g_1(x)+r_1(x)=g_2(x)+r_2(x)$ where $g_1(x),g_2(x),r_1(x),r_2(x)$ satisfy conditions (i), (ii), (iii) of Proposition 2.6. We have $r_1(x)-r_2(x)=g_2(x)-g_1(x)\in R$. If $r_1(x)-r_2(x)\not= 0$ then $\mathrm o(r_1(x)-r_2(x))\notin  \langle \mathrm o(f_1),\ldots,\mathrm o(f_s)\rangle$. This contradicts the hypothesis.
 \end{proof}

Let, as above, ${R}=\KK\llbracket f_1,\ldots,f_s\rrbracket $. We shall suppose that $f_i$ is monic for all $1\leq i\leq s$. Define 
$$
\phi:\KK[X_1,\ldots,X_s]\longrightarrow \KK[x],\ \phi(X_i)=\mathrm M_\mathrm o(f_i) \hbox{ for all }i\in\{1,\ldots, s\}.$$ 
Let $\lbrace F_1,\ldots,F_r\rbrace$ be a generating system of the kernel of $\phi$. We can choose all of them to be binomials. If $F_i=X_1^{\alpha^i_1}\cdots  X_s^{\alpha^i_s}-X_1^{\beta^i_1}\cdots  X_s^{\beta^i_s}$, we set $S_i=f_1^{\alpha^i_1}\cdots  f_s^{\alpha^i_s}-f_1^{\beta^i_1}\cdots  f_s^{\beta^i_s}$. Note that if $p=\sum_{k=1}^s\alpha^i_k \mathrm o(f_k)=\sum_{k=1}^s\beta^i_k \mathrm o(f_k)$, then $\mathrm o(S_i)>p$.

\begin{thm}\label{carac-local}
The system $\lbrace f_1,\ldots,f_s\rbrace$ is a basis of ${R}$ if and only if $\mathrm R_\mathrm o(S_i,\lbrace f_1,\ldots,f_s\rbrace)=0$ for all $i\in\{1,\ldots,r\}$.
 \end{thm}
 \begin{proof} Suppose that $\lbrace f_1,\ldots,f_s\rbrace$ is a basis of ${R}$. Since $S_i\in{R}$ for all $i\in\{1,\ldots,r\}$, then, by Proposition \ref{carac-local-1}, $\mathrm R_\mathrm o(S_i,\lbrace f_1,\ldots,f_s\rbrace)=0$. 
 
For the sufficiency assume to the contrary that $\{f_{1},\ldots, f_{s}\}$ is not a basis of ${R}$. Then there exists $f\in {R}$ such that $\mathrm o(f)\not\in \langle \mathrm o(f_{1}),\ldots, \mathrm o(f_{s})\rangle$. %Conversely, by Proposition \ref{carac-local-1}, we need to prove that $\mathrm R_\mathrm o(f,\lbrace f_1,\ldots,f_s\rbrace)=0$ for all $f\in{R}$. Let to this end $f\in {R}$ and assume that $f\not=0$. 
Write 
$$
 f=\sum_{\underline{\theta}}c_{\underline{\theta}}f_1^{\theta_1}\cdots  f_s^{\theta_s}.
$$
For all $\underline{\theta}$, if $c_{\underline{\theta}}\not=0$, we set $p_{\underline{\theta}}=\sum_{i=1}^s\theta_i\mathrm o(f_i)=\mathrm o(f_1^{\theta_1}\cdots  f_s^{\theta_s})$. Let $p=\mathrm{min}\lbrace p_{\underline{\theta}}\mid c_{\underline{\theta}}\not=0\rbrace$ and let $\lbrace \underline{\theta}^1,\ldots,\underline{\theta}^l\rbrace$ be such that $p=\mathrm o(f_1^{\theta^i_1}\cdots  f_s^{\theta^i_s})$ for all
 $i\in\{1,\ldots, l\}$ (such a set is clearly finite). Also $p\le \mathrm o(f)<\infty$.
 
If $\sum_{i=1}^lc_{\underline{\theta}^i}\mathrm M_{\mathrm o}(f_1^{\theta^i_1}\cdots  f_s^{\theta^i_s})\not=0$, then $p=\mathrm o(f)\in \langle \mathrm o(f_1),\ldots,\mathrm o(f_s) \rangle$. But this is impossible. Hence,  $\sum_{i=1}^lc_{\underline{\theta}^i}\mathrm M_{\mathrm o}(f_1^{\theta^i_1}\cdots  f_s^{\theta^i_s})=0$, and then 
 $\sum_{i=1}^lc_{\underline{\theta}^i}X_1^{\theta^i_1}\cdots  X_s^{\theta^i_s}\in \ker(\phi)$. Hence
 $$
 \sum_{i=1}^lc_{\underline{\theta}^i}X_1^{\theta^i_1}\cdots  X_s^{\theta^i_s}=\sum_{k=1}^r\lambda_kF_k
 $$
 with $\lambda_k\in\KK[X_1,\ldots,X_s]$ for all $k\in\{1,\ldots,r\}$ (recall that $F_{1},\ldots, F_{r}$ are binomials generating $\ker(\phi)$). This implies that 
 $$
 \sum_{i=1}^lc_{\underline{\theta}^i}f_1^{\theta^i_1}\cdots  f_s^{\theta^i_s}=\sum_{k=1}^r\lambda_k(f_1,\ldots,f_s)S_k.
 $$
From the hypothesis $\mathrm R_\mathrm o(S_k,\lbrace f_1,\ldots,f_s\rbrace)=0$. Hence there is an expression of $S_k$ of the form $S_k=\sum_{\underline{\beta}^k}c_{\underline{\beta}^k} f_1^{\beta^k_1}\cdots f_s^{\beta^k_s}$ with $\mathrm o(f_1^{\beta^k_1}\cdots f_s^{\beta^k_s})\ge \mathrm o(S_k)$. %Also from the definition of $S_k$, $\mathrm o(S_k)>p$. 

So by replacing $ \sum_{i=1}^lc_{\underline{\theta}^i}f_1^{\theta^i_1}\cdots  f_s^{\theta^i_s}$ with $\sum_{k=1}^r\lambda_k(f_1,\ldots,f_s)\sum_{\underline{\beta}^k}c_{\underline{\beta}^k} f_1^{\beta^k_1}\cdots f_s^{\beta^k_s}$ in the expression of $f$, we can rewrite $f$ as $f=\sum_{\underline{\theta}'}c_{\underline{\theta}'}f_1^{\theta'_1}\cdots  f_s^{\theta'_s}$ with $\min\lbrace \mathrm o(f_1^{\theta'_1}\cdots  f_s^{\theta'_s})\mid c_{\underline{\theta}'}\not=0\rbrace>p$. 
 
 Since $\mathrm o(f)<+\infty$, this process will stop, yielding a contradiction.
\end{proof}

 \begin{alg}
 Let the notations be as above.
 
 \begin{enumerate}[1.]
 
 \item If $\mathrm R_\mathrm o(S_k,\lbrace f_1,\ldots,f_s\rbrace)=0$ for all $k\in\{1,\ldots,r\}$, then $\lbrace f_1,\ldots,f_s\rbrace$ is a basis of ${R}$.
 
 \item If $r(x)=\mathrm R_\mathrm o(S_k,\lbrace f_1,\ldots,f_s\rbrace)\not= 0$ for some $k\in\{1,\ldots,r\}$, and if $\mathrm M_\mathrm o(r(x))=ax^{q}$, then we set $f_{s+1}=\frac{1}{a}r(x)$, and we restart with $\lbrace f_1,\ldots,f_{s+1}\rbrace$. Note that in this case, \[\langle \mathrm o(f_1),\ldots,\mathrm o(f_s)\rangle\subsetneq \langle\mathrm o(f_1),\ldots,\mathrm o(f_s),\mathrm o(f_{s+1}) \rangle\subseteq \mathrm o(R).\] 
 
 \end{enumerate}
 This process will stop, giving a basis of ${R}$, because the complement of $\mathrm o(R)$ in $\NN$ is finite.
 
 Observe that $r(x)$ is not in general a polynomial. So we must use a trick to compute it, or at least the relevant part of it. This is accomplished by using Proposition \ref{a}. If in the current step of the algorithm $\langle \mathrm o(f_1),\ldots, \mathrm o(f_s)\rangle$ is a numerical semigroup, then we compute its conductor, say c. Then $c\ge \mathrm c(\mathrm o(R))$. To compute $\mathrm R_\mathrm o(f,\{f_1,\ldots,f_s\})$ we do the following. Let $p=\mathrm o(f)$.
 \begin{enumerate}[1.]
 \item If $p\ge c$, then return 0. We implicitly assume that $x^a$ is in our generating set for $a\in c+\NN$ (though we do not store them).
 \item If $p\in \langle \mathrm o(f_1),\ldots, \mathrm o(f_s)\rangle$, then $\mathrm M_\mathrm o(f)=\sum_{\underline{\theta}} c_{\underline{\theta}}\mathrm M_\mathrm o (f_1)^{\theta_1}\cdots \mathrm M_\mathrm o (f_s)^{\theta_s}$. Set $f=f-\sum_{\underline{\theta}} c_{\underline{\theta}} f_1^{\theta_1}\cdots f_s^{\theta_s}$, and call recursively $\mathrm R_\mathrm o(f,\{f_1,\ldots, f_s\})$ (the process will stop because the order of the new $f$ is larger, and eventually will become bigger than $c$ after a finite number of steps). 
 \item If $p\not\in \langle \mathrm o(f_1),\ldots, \mathrm o(f_s)\rangle$, then return $f$. 
 \end{enumerate}
 
 If $\langle \mathrm o(f_1),\ldots, \mathrm o(f_s)\rangle$ is not a numerical semigroup, let $d$ be its greatest common divisor. Set $c=d\mathrm c(\langle \mathrm o(f_1),\ldots, \mathrm o(f_s)\rangle/d)$. In this case we proceed as follows. 
 \begin{enumerate}[1.]
 \item If $p\ge c$, then return $f$. We cannot ensure here that $f$ will be reduced to zero, so we add it just in case.
 \item If $p\in \langle \mathrm o(f_1),\ldots, \mathrm o(f_s)\rangle$, then $\mathrm M_\mathrm o(f)=\sum_{\underline{\theta}} c_{\underline{\theta}}\mathrm M_\mathrm o (f_1)^{\theta_1}\cdots \mathrm M_\mathrm o (f_s)^{\theta_s}$. Set $f=f-\sum_{\underline{\theta}} c_{\underline{\theta}} f_1^{\theta_1}\cdots f_s^{\theta_s}$, and call recursively $\mathrm R_\mathrm o(f,\{f_1,\ldots, f_s\})$. 
 \item If $p\not\in \langle \mathrm o(f_1),\ldots, \mathrm o(f_s)\rangle$, then return $f$. One might check first if $d$ does not divide $p$, because in this case for sure $p\not\in \langle \mathrm o(f_1),\ldots, \mathrm o(f_s)\rangle$.
 \end{enumerate}

Observe that by adding the conditions $p\ge c$, we are avoiding entering in an eventual infinite loop.
\end{alg}
 
%
% 
% \noindent The algorithm above can be refined this way. Let the notations be as above and let $E_0=\lbrace \mathrm o(f_1),\ldots,\mathrm o(f_s)\rbrace$ and $R_0=\lbrace \mathrm R_\mathrm o(S_i,\lbrace f_1,\ldots,f_s\rbrace),i=1,\ldots,r\rbrace$. Suppose that $R_0\not=\lbrace 0\rbrace$, and let $p=\min\lbrace \mathrm o(f)\mid f\in \mathrm R_\mathrm o\rbrace$, then $p=\min(\mathrm o({R})\setminus \langle \mathrm o(f_1),\ldots,\mathrm o(f_s) \rangle)$. If $p=\mathrm o(f)$ for some  monic element $f\in {R}$, then we set $f=f_{s+1}$ and we restart with $\lbrace f_1,\ldots,f_{s+1}\rbrace$.  
% 
% 
 
 \medskip
 
Suppose that $\lbrace f_1,\ldots,f_s\rbrace$ is a basis of ${R}$. Also suppose that for all $i\in\{1,\ldots,s\}$, $f_i$ is monic. We say that $\lbrace f_1,\ldots,f_s\rbrace$ is  a \emph{minimal basis} of ${R}$ if $\mathrm o(f_1),\ldots,\mathrm o(f_s)$ generate minimally the semigroup $\mathrm o({R})$. We say that $\lbrace f_1,\ldots,f_s\rbrace$ is a \emph{reduced basis} of ${R}$ if $\mathrm{supp}(f_i(x)-\mathrm M_\mathrm 0(f_i))\subseteq \NN\setminus \mathrm o({R})$. Let $i\in\{1,\ldots,s\}$. If $\mathrm o(f_i)\in\langle \mathrm o(f_1),\ldots,\mathrm o(f_{i-1}),\mathrm o(f_{i+1}),\ldots,\mathrm o(f_s) \rangle$, then $\lbrace f_1,\ldots,f_{i-1},f_{i+1},\ldots,f_s\rbrace$ is also a basis of ${R}$.
Furthermore, by applying the division process of Proposition \ref{remainder} to $f_i-\mathrm M_\mathrm o(f_i)$, we can always construct a reduced basis of ${R}$.
 
\begin{cor}
The algebra ${R}$ has a unique minimal reduced  basis.
\end{cor}
\begin{proof} Let $\lbrace f_1,\ldots,f_s\rbrace$ and $\lbrace g_1,\ldots,g_{s'}\rbrace$ be two minimal reduced bases of ${R}$. Hence $s$ is the embedding dimension of $\mathrm o({R})$, and the same holds for $s'$; whence they are equal. Let $i=1$. There exists $j_1$ such that $\mathrm o(f_1)=\mathrm o(g_{j_1})$, because minimal generating systems of numerical semigroups are unique. If $f_1-g_{j_1}\not= 0$, then $\mathrm o(f_1-g_{j_1})\notin \mathrm o({R})$ (the basis is reduced), which is a contradiction because $f_1-g_{j_1}\in{R}$. The same argument shows that $\lbrace f_1,\ldots,f_s\rbrace=\lbrace g_1,\ldots,g_{s}\rbrace$
\end{proof}

 \begin{rem}{\rm  Let ${R}=\KK\llbracket f_1,\ldots,f_s\rrbracket $ and assume that $f_i$ is monic for all $1\leq i\leq s$. Also assume that $\mathrm o(f_1)\leq \mathrm o(f_2)\leq\ldots\leq \mathrm o(f_s)$. Set $n=\mathrm o(f_1)$ and let $f_1=x^{n}+\sum_{i> n}c^1_ix^i$. By an analytic change of variables, we may assume that $f_1=x^n$, hence, up to an analytic isomorphism, we may assume that ${R}=\KK\llbracket x^n,f_2,\ldots,f_s\rrbracket $. In particular,  we may assume that ${R}$ has a minimal reduced basis of the form ${x^n, g_2(x),\ldots,g_{s'}(x)}$. 
 }
 \end{rem}
%\subsection{Examples and applications}
% 
%In this section we use the algorithm given above in some examples and applications.
 
\begin{ex}\label{io}
Let ${R}=\KK\llbracket x^{4}+x^{5},x^{6},x^{15}+x^{16}+ \sum_{n\ge 20} x^n\rrbracket $, with $\KK$ a field
of characteristic zero. Then ${R}$ is a
one-dimensional ring. Since  the conductor of $\langle 4,6,15\rangle$ is 18, we have that $\lambda_{R}(\KK\llbracket X\rrbracket /{R})<\infty$ and we know by Propsition \ref{a} that $R=\KK\llbracket x^{4}+x^{5},x^{6},x^{15}+x^{16}\rrbracket$. 
Let us denote $x^{4}+x^{5}$
by $f_{1}$, $x^{6}$ by $f_{2}$ and $x^{15}+x^{16}$ by $f_{3}$. The kernel of $\phi: \KK[X_1,X_2,X_3]\longrightarrow\KK[x]$, $\phi(X_1)=x^4$, $\phi(X_2)=x^6$ and $\phi(X_3)=x^{15}$ is generated by $X_1^3-X_2^2, X_3^2-X_2^5$, hence we get $S_1= x^{13}+x^{14}+\frac{1}3 x^{15}$ and $S_2=x^{31}+\frac{1}2 x^{32}$. As $13\not\in \langle 4,6,15\rangle$, we add it as $f_4=S_1$. We do not care about $S_2$, because the conductor of $\langle 4,6,15\rangle$ is 18.

Now, the conductor of $\langle 4,6,13,15\rangle$ is $12$. If we compute a system of generators of the kernel of $\varphi: \KK[X_1,X_2,X_3, X_4]\longrightarrow\KK[x]$, $\phi(X_1)=x^4$, $\phi(X_2)=x^6$, $\phi(X_3)=x^{15}$ and $\phi(X_4)=x^{13}$, then all the elements $S_i$ have orders greater than $12$, and so the algorithm ends. We conclude that $\mathrm o(R)=\langle 4,6,13,15\rangle$.

We have implemented this algorithm in the \texttt{numericalsgps} (\cite{numericalsgps}) \texttt{GAP} (\cite{gap}) package. Next we illustrate how to compute this semigroup with the functions we have implemented (that will be available in the next release of the package).

\begin{verbatim}
gap> x:=X(Rationals,"x");;
gap> l:=[x^4+x^5,x^6,x^15+x^16];;
gap> s:=SemigroupOfValuesOfCurve_Local(l);;
gap> MinimalGeneratingSystem(s);
[ 4, 6, 13, 15 ]
gap> SemigroupOfValuesOfCurve_Local(l,13);
x^13
\end{verbatim}

\end{ex}
 
\begin{rem}\label{e}
It is known (cf. \cite[Section II.1]{1}) that there exist relations between algebraic characteristics and invariants of the semigroup $\mathrm o({R})$ and the ring ${R}$. Hence, in the Example \ref{io}, from $\mathrm o({R})=\langle 4,6,13,15 \rangle=\{0,4,6,8,10,12,\longrightarrow\}$, we deduce that the conductor of $\mathrm o(R)$ is 12. The conductor of $R$ in $\bar{R}$ is precisely $(x^{\mathrm c(\mathrm o(R))})$ (\cite{1}). We have that $\lambda_{R}(\bar {R}/{R})=|[0,\mathrm c(\mathrm o({R}))-1]\cap (\NN\setminus \mathrm o({R}))|=7$ counts the number of gaps of $\mathrm o(R)$ (\cite[Proposition 1]{4}). The integer $\lambda_{R}(\bar {R}/{R})$ is the degree of singularity of $R$, and this is why the genus of $\mathrm o(R)$ is called by some authors the degree of singularity of the semigroup (see \cite{1,4}). In the setting of Weirstrass numerical semigroups the genus coincides with the geometrical genus of the curve used to define the semigroup (\cite{centina}).

The number of sporadic elements of $\mathrm o(R)$ is $\lambda_{R}({R}/({R}:\bar {R}))=|[0,\mathrm c(\mathrm o({R}))-1]\cap \mathrm o({R})|=5$. 

The ring $R$ is Cohen-Macaulay and its type is less than or equal to $\#\mathrm T(\mathrm o({R}))=3$, where for a numerical semigroup $\Gamma$, $\mathrm T(\Gamma)=\{x\in \mathbb Z\setminus \Gamma \mid x+\Gamma^*\subseteq \Gamma\}$ (\cite[Proposition II.1.16]{1}).  According also to \cite[Proposition II.1.16]{1},  equality holds if and only if $\mathrm o(R:\mathfrak m) = \mathrm T(\mathrm o(R))$. However $\mathrm T(\mathrm o(R))=\{2,9,11\}$ and $2\not\in \mathrm o(R:\mathfrak m)$. So in our example we get an strict inequality.

%\textcolor{red}{should not be this an equality according to Theorem 1 in \cite{4}?}
\end{rem}

\begin{ex}
Let ${R}=\KK\llbracket x^{4},x^{6}+x^{7},x^{13}+a_{14}x^{14}+a_{15}x^{15}+\dots\rrbracket $
with $\KK$ a field. Using the same argument as in the Example \ref{io}, we find that if char $\KK\ne 2$, we have that if $a_{15}-a_{14}+1/2 =0$, then $\{x^{4},x^{6}+x^{7},x^{13}\}$ is the reduced basis of ${R}$. Furthermore, since $\langle 4,6,13\rangle$ is a symmetric numerical semigroup
(the number of nonnegative integers not in the semigroup equals the conductor divided by two), then, by \cite{3}, ${R}$ is Gorenstein. Finally $\lambda_{R}(\bar{R}/{R})=8$. Otherwise if $a_{15}-a_{14}+1/2 \neq 0$, then
$\{x^{4},x^{6}+x^{7},x^{13},x^{15}\}$ is the reduced basis of ${R}$
with $R$ a non Gorenstein ring. Furthermore $\lambda_{R}(\bar{R}/{R})=7$.

Otherwise, if char $\KK=2$, then the reduced basis of ${R}$ is $\{x^{4},x^{6}+x^{7}, x^{13},x^{15}\}$ and ${R}$ is not a Gorenstein ring. Here, $\lambda_{R}(\bar{R}/{R})=7$.
\end {ex}

%\begin{ex}\label{io2}
%Let ${R}=\KK\llbracket X^{4}+\sum_{i\geq 5}a_{i}X^{i},X^{5}+\sum_{i\geq 6}b_{i}X^{i},
%x^{6}+\sum_{i\geq 7}c_{i}x^{i}\rrbracket $, with $\KK$ a field.
%Using the same argument as in the Example
%\ref{io}, we have that
%$\{x^4+l_1 x^7,x^5+l_2 x^7,x^6+l_3 x^7\}$ is the reduced basis of the
%Gorenstein ring ${R}$ for some $l_{1},l_{2},l_{3}$. Furthermore, we have
%$\lambda_{R}(\bar{R}/{R})=4$.
%\end{ex}

\begin{ex}
Let ${R}=\KK\llbracket x^{8},x^{12}+x^{14}+x^{15}\rrbracket $, with $\KK$ a field of
characteristic zero.
Using the same argument as in the Example
\ref{io}, we have that
\begin{multline*}
\{x^8,x^{12}+x^{14}+x^{15},x^{26}+x^{27}+x^{29}-\frac{1}{2} x^{31},\\
x^{53}+\frac{1}{2}x^{55}-\frac{1}{2}x^{57}-
\frac{1}{8}x^{63}+
\frac{25}{8}x^{67}-\frac{95}{32}x^{71}-
\frac{15}{16}x^{75}-\frac{135}{32}x^{83}\}
\end{multline*}        is the reduced basis of the
Gorenstein ring ${R}$. Furthermore, we have $\lambda_{{R}}(\bar{R}/{R})=42$.
\begin{verbatim}
gap> l:=[x^8,x^12+x^14+x^15];;
gap> SemigroupOfValuesOfCurve_Local(l,"basis");
[ x^8, x^15+x^14+x^12, -1/2*x^31+x^29+x^27+x^26, 
-135/32*x^83-15/16*x^75-95/32*x^71+25/8*x^67-1/8*x^63-1/2*x^57+1/2*x^55+x^53 ]
\end{verbatim}
\end{ex}
\begin{ex}
The following battery of examples was provided by Lance Bryant as a test for our algorithm.
\begin{verbatim}
gap> l:=[ [ x^6,x^8+x^9,x^19], [x^7,x^9+x^10,x^19,x^31], [x^7,x^21+x^28+x^33], 
[x^4,x^6+x^7,x^13], [x^6,x^8+x^11,x^10+2*x^13,x^21], [x^5,-x^18-x^21,-x^23,-x^26], 
[x^5,-x^18-x^21,-x^26], [x^5,-x^18-x^21,x^23-x^26], [x^6,x^9+x^10,x^19], 
[x^7,x^9+x^10,x^19], [x^8,x^9+x^10,x^19], [x^7,x^9+x^10,x^17,x^19] ] ;;
gap> List(l, i->MinimalGeneratingSystem(SemigroupOfValuesOfCurve_Local(i)));
[ [ 6, 8, 19, 29 ], [ 7, 9, 19, 29, 31 ], [ 7, 33 ], [ 4, 6, 13, 15 ], 
  [ 6, 8, 10, 21, 23, 25 ], [ 5, 18, 26, 39, 47 ], [ 5, 18, 26, 39, 47 ], 
  [ 5, 18, 26, 39, 47 ], [ 6, 9, 19, 20 ], [ 7, 9, 19, 29 ], [ 8, 9, 19, 30 ], 
  [ 7, 9, 17, 19, 29 ] ]
\end{verbatim}
\end{ex}

\begin{rem}
We do not know a priori if $\bar{R}\neq\KK[\![x]\!]$ and we do not have a general procedure to check it. If the algorithm is called with such an $R$, it will eventually not stop.
\end{rem}
 
\section{Deformation to a toric ideal}\label{deformation-local}
Let the notations be as in Section \ref{formal-space}. Given $f(x)=\sum_{i\geq p}c_ix^i\in\KK\llbracket x\rrbracket $, we set $H_f(u,x)=\sum_{i\geq p}c_iu^{i-p}x^i$. In particular, if we consider the linear form $L:\NN^2\longrightarrow \NN, L(a,b)=b-a$, then $H_f$ is $L$-homogeneous of order $p$, that is, $L(i-p,i)=p$ for all $i\in\mathrm{supp}(f)$. We set $H_{R}=\KK\llbracket u, H_f \mid f\in{R}\rrbracket $. With these notations we have the following.
 
\begin{prop} \label{homogeneous-local}
The set $\lbrace f_1,\ldots,f_s\rbrace$ is a basis of ${R}$ if and only if $H_{R}=\KK\llbracket u,H_{f_1},\ldots,H_{f_s}\rrbracket $.
\end{prop}
\begin{proof} Note first that if $H_g\in H_{R}$ for some $g\in \KK\llbracket x\rrbracket$, then $g\in R$. Suppose that $\lbrace f_1,\ldots,f_s\rbrace$ is a basis of ${R}$ and let $f(x)\in{R}$. Write
$H_f(u,x)=\sum_{i\geq p}c_iu^{i-p}x^i$. We have $\mathrm M_\mathrm o(f)=c_px^p=c_p\prod_{i=1}^s\mathrm M_\mathrm o(f_i)^{p^k_i}$, hence
$$
H_f-c_p\prod_{i=1}^sH_{f_i}^{p^k_i}=u^qH_{f^1}
$$
with $f^1\in {R}$ and either $f^1=0$, or $\mathrm o(f^1)>p$. In the second case we restart with $f^1$. A similar argument as in Proposition 2.6 proves our assertion.

Conversely, suppose that $H_{R}=\KK\llbracket u, H_{f_1},\ldots,H_{f_s}\rrbracket $ and let $f\in{R}$. Let $P(X_0,X_1,\ldots,X_s)\in\KK\llbracket X_0,X_1,\ldots,X_s\rrbracket $ such that $H_f=P(u,H_{f_1},\ldots,H_{f_s})$. Setting $u=0$, we get that  $\mathrm M_\mathrm o(f)=P(\mathrm M_\mathrm o(f_1),\ldots,\mathrm M_\mathrm o(f_s))\in \KK\llbracket \mathrm M_\mathrm o(f_1),\ldots,\mathrm M_\mathrm o(f_s)\rrbracket $. Hence $\mathrm M_\mathrm o(f)\in \KK[\mathrm M_\mathrm o(f_1),\ldots,\mathrm  M_\mathrm o(f_s)]$.
\end{proof}

Suppose that $\lbrace f_1,\ldots,f_s\rbrace$ is a basis of ${R}$. Then $T=\KK\llbracket u\rrbracket \llbracket H_{f_1},\ldots,H_{f_s}\rrbracket $  is a $\KK\llbracket u\rrbracket$-module. When $u=1$ (respectively $u=0$), we get $T\mid_{u=1}={R}$ (respectively $T\mid_{u=0}=\KK\llbracket M(f_1),\ldots,M(f_s)\rrbracket $). Hence we get a deformation from ${R}$ to $\KK\llbracket \mathrm M_\mathrm o(f_1),\ldots,\mathrm M_\mathrm o(f_s)\rrbracket $. More precisely let 
$$
\psi:\KK\llbracket X_1,\ldots,X_s\rrbracket \longrightarrow R=\KK\llbracket {f_1},\ldots,{f_s}\rrbracket 
$$
and 
$$
H_{\psi}:\KK\llbracket u\rrbracket \llbracket X_1,\ldots,X_s\rrbracket \longrightarrow T=\KK\llbracket u\rrbracket \llbracket H_{f_1},\ldots,H_{f_s}\rrbracket 
$$
be the morphisms of rings such that $H_{\psi}(u)=u$, $\psi(X_i)=f_i$ and $H_{\psi}(X_i)=H_{f_i}$ for all $i\in \{1,\ldots,s\}$. Let,  as in Section \ref{formal-space}, 
$$
S_i=f_1^{\alpha^i_1}\cdots  f_s^{\alpha^i_s}-f_1^{\beta^i_1}\cdots  f_s^{\beta^i_s}=\sum_{\underline{\theta}^i}c^i_{\underline{\theta}^i}f_1^{\theta^i_1}\cdots  f_s^{\theta_s^i}, 1\leq i\leq r
$$ 
with $\mathrm o(f_1^{\theta^i_1}\cdots  f_s^{\theta_s^i})=D^i_{\underline{\theta}^i}>\sum_{k=1}^s\alpha^i_k\mathrm o(f_k)=\sum_{k=1}^s\beta^i_k\mathrm o(f_k)=p_i$. Let $I$ (respectively $J$) be the ideal generated by $(G_i=X_1^{\alpha^i_1}\cdots  X_s^{\alpha^i_s}-X_1^{\beta^i_1}\cdots  X_s^{\beta^i_s}-\sum_{\underline{\theta}^i}c^i_{\underline{\theta}^i}X_1^{\theta^i_1}\cdots  X_s^{\theta_s^i})_{1\leq i\leq r}$ (respectively $(H_i=X_1^{\alpha^i_1}\cdots  X_s^{\alpha^i_s}-X_1^{\beta^i_1}\cdots  X_s^{\beta^i_s}-\sum_{\underline{\theta}^i}u^{D^i_{\underline{\theta}^i}-p_i}c^i_{\underline{\theta}^i}X_1^{\theta^i_1}\cdots  X_s^{\theta_s^i})_{1\leq i\leq r}$) in $\KK\llbracket X_1,\ldots,X_s\rrbracket $ (respectively $\KK\llbracket u\rrbracket \llbracket X_1,\ldots,X_s\rrbracket $). 

Well shall consider on ${\mathbb N}^s$ (respectively, $\mathbb{N}^{s+1}$) the linear form 
\[\mathrm O(\theta_1,\ldots,\theta_s)=\sum_{i=1}^s\theta_i\mathrm o(f_i)\] 
(respectively $\mathrm O_h(\theta_0,\theta_1,\ldots,\theta_s)=-\theta_0+\sum_{i=1}^s\theta_i \mathrm o(f_i)$).

Given a monomial $X_1^{\theta_1}\cdots X_s^{\theta_s}$ (respectively $u^{\theta_0} X_1^{\theta_1}\cdots X_s^{\theta_s}$), we set $\mathrm O(X_1^{\theta_1}\cdots X_s^{\theta_s})=\mathrm O(\theta_1,\ldots,\theta_s)$ (respectively $\mathrm O_h(u^{\theta_0} X_1^{\theta_1}\cdots X_s^{\theta_s})=\mathrm O_h(\theta_0,\theta_1,\ldots,\theta_s)$). 

For $G=\sum_{\theta} c_{\theta}X_1^{\theta_1}\cdots X_s^{\theta_s}$ (respectively $H=\sum_{\theta} c_{\theta}u^{\theta_0}X_1^{\theta_1}\cdots X_s^{\theta_s}$), we say that $G$ (respectively $H$) is $\mathrm O$-homogeneous of order $a$ (respectively $\mathrm O_h$-homogeneous of order $b$) if  $\mathrm O(\theta_1,\ldots,\theta_s)=a$ (respectively $\mathrm O_h(\theta_0,\theta_1,\ldots,\theta_s)=b$) for all $(\theta_1,\ldots,\theta_s)$ (respectively $(\theta_0,\theta_1,\ldots,\theta_s)$) such that $c_{\theta}\not=0$. More generally let $G=\sum_{k\geq 0}G_{p_k}$ where $p_0<p_1<\cdots$ and $G_{p_k}$  is $\mathrm O$-homogeneous of order $p_k$. We set $\mathrm O(G)=p_0$. We also set $\mathrm{in}(G)=G_{p_0}$ and we call it the \emph{initial form} of $G$. 

We finally set $\mathrm O(0)=+\infty$, and we recall that $\mathrm O(G)=+\infty$ if and only if $G=0$.
%$\bar{S}_i=F_i+\sum_{\underline{\theta}^i}c^i_{\underline{\theta}^i}X_1^{\theta^i_1}\cdots  X_s^{\theta_s^i}$. Note %that $F_i$ is quasi-homogeneous with respect to the linear form $L(X_i)=\mathrm o(f_i)$. Furthermore, if
%$c^i_{\underline{\theta}^i}\not=0$, then $L(X_1^{\theta^i_1}\cdots  X_s^{\theta_s^i})>L(F_i)$. Let $L(F_i)=D_i$ and let

%$$
%H_L(\bar{S}_i)=F_i+\sum_{\underline{\theta}^i}c^i_{\underline{\theta}^i}u^{L(X_1^{\theta^i_1}\cdots  X_s^{\theta_s^i})- %D_i}X_1^{\theta^i_1}\cdots  X_s^{\theta^i_s}.
%$$

\begin{lem}\label{localkernel} With the standing notations and hypothesis, the kernel of $\psi$  is generated by $I$. 
\end{lem}
\begin{proof} 
Let, as in Section \ref{formal-space}, $F_1,\ldots, F_r$ be a generating system of the kernel of the morphism \[\phi: \KK\llbracket X_1,\ldots,X_s\rrbracket \longrightarrow \KK\llbracket x\rrbracket,\ \phi(X_i)=\mathrm M_\mathrm o(f_i)\] for all $i\in \lbrace 1,\ldots,s\rbrace$. In particular $F_i$ is $\mathrm O$-homogeneous of order $\mathrm o(f_i)$ for all $i\in\lbrace 1,\ldots, s\rbrace$, and ${\KK\llbracket X_1,\ldots,X_s\rrbracket}/{(F_1,\ldots,F_r)}\simeq \KK\llbracket \mathrm M_\mathrm o(f_1),\ldots,\mathrm  M_\mathrm o(f_s)\rrbracket$. 

For all $i\in\lbrace 1,\ldots, r\rbrace, \psi(G_i)=0$. Hence $I\subseteq \ker(\psi)$. 

For the other inclusion, let $G=\sum_{\theta}c_{\theta}X_1^{\theta_1}\cdots X_s^{\theta_s}\in \ker (\psi)$. Write $G=\sum_{k\geq 0}c_{\theta^k}X_1^{\theta^k_1}\cdots X_s^{\theta^k_s}$ where $\mathrm O(\theta^0_1,\ldots,\theta^0_s)\leq\mathrm  O(\theta^{1}_1,\ldots,\theta^1_s)\leq\cdots$. Since $\psi(G)=0$, we have that 
\[
\sum_{k\geq 0}c_{\theta^k}f_1^{\theta^k_1}\cdots f_s^{\theta^k_s}=0.
\]
In particular $\sum_{k, O(\theta^k)=O(\theta^0)}c_{\theta^k}\mathrm M_o(f_1)^{\theta^k_1}\cdots \mathrm  M_o(f_s)^{\theta^k_s}=0$, and consequently $\sum_{k, O(\theta^k)=O(\theta^0)}c_{\theta^k}X_1^{\theta^k_1}\cdots X_s^{\theta^k_s}\in \ker(\phi)$. This implies that 
\[
\sum_{k, O(\theta^k)=O(\theta^0)}c_{\theta^k}X_1^{\theta^k_1}\cdots X_s^{\theta^k_s}=\sum_{i=1}^r\lambda^0_iF_i
\]
for some $\lambda^0_i\in \KK\llbracket X_1,\ldots,X_s\rrbracket$, $i\in\lbrace 1,\ldots,r\rbrace$, with $\lambda^0_i$ is $\mathrm O$-homogeneous of order $\mathrm O(G)-\mathrm O(F_i)$. Hence
\[
\sum_{k, O(\theta^k)=O(\theta^0)}c_{\theta^k}f_1^{\theta^k_1}\cdots f_s^{\theta^k_s}=\sum_{i=1}^r\lambda^0_i(f_1,\ldots,f_s)S_i.
\]
Let $G^1=G-\sum_{i=1}^r\lambda^0_iG_i$. We have $G^1\in \ker(\psi)$. If $G^1\not=0$, then $\mathrm O(G)<\mathrm O(G^1)$. Then we restart with $G^1$. We construct  in the same way 
$G^2$, and $\lambda^1_1,\ldots,\lambda^1_r$ such that $G^1=G^2+\sum_{j=1}^r\lambda^1_jG_j$ with $\mathrm O(G)<\mathrm O(G^1)<\mathrm  O(G^2)$, $\lambda^1_i$ $\mathrm O$-homogeneous and  $\mathrm O(\lambda^0_i)<\mathrm O(\lambda^1_i)$ for all $i\in\lbrace 1,\ldots,r\rbrace$. If we continue in this way, we get that for all $k\geq 0$, 
\[
G=G^{k+1}+\sum_{i=1}^r(\lambda^0_i+\lambda^1_i+\ldots+\lambda^k_i)G_i,
\]
with $\mathrm O(G)<\mathrm O(G^1)<\ldots<\mathrm O(G^{k+1})$, $\lambda_i^j$  $O$-homogeneous, and $\mathrm O(\lambda^0_i)<\mathrm O(\lambda^1_i)<\cdots<\mathrm O(\lambda^k_i)$ for all $i\in\{ 1,\ldots,r\}$ and for all $j\in\lbrace 1,\ldots,k\rbrace$. If $G^k=0$ for some $k$, then we are done. Otherwise, let $\lambda_i=\sum_{k=0}^{\infty}\lambda^k_i$, and let $\bar{G}=\lim_{k\to +\infty}G^k$. We have $\bar{G}=0$ and $G=\sum_{i=1}^r\lambda_iG_i$. This proves our assertion.
\end{proof}

 Let the notations be as above. Let $G= \sum_{\theta}c_{\theta} X_1^{\theta_1}\cdots X_s^{\theta_s} \in \KK\llbracket X_1,\ldots,X_s\rrbracket$ and write  $G=\sum_{i\geq 0}{G_{p_i}}$ with $p_0<p_1<\cdots$ and $G_{p_i}$  $\mathrm O$-homogeneous. We set $H_G=\sum_{i\geq 0}u^{p_i-p_0}G_{p_i}$, in such a way that $H_G$ is $O_h$-homogeneous of order $p_0$. Given an ideal $S$ of $\KK\llbracket X_1,\ldots,X_s\rrbracket$, we set $\mathrm{in}(S)=(\mathrm{in}(G) \mid G\in S\setminus\{0\})$. We also denote by $H_S=(H_G \mid G\in S\setminus\{0\})\KK\llbracket u,X_1,\ldots,X_s\rrbracket$. With these notations we have $\mathrm{in}(S_i)=F_i$ and $H_{G_i}=H_i$ for all $i\in\lbrace 1,\ldots,r\rbrace$.

\begin{lem}\label{standardbasis} Let the notations be as above. We have $\mathrm{in}(I)=(F_1,\ldots,F_r)$, and $H_I=(H_1,\ldots,H_r)=J$.
\end{lem}
\begin{proof} The first assertion follows from the proof of Lemma \ref{localkernel}. To prove the second assertion, let $H\in H_I$ and assume that $H$ is $\mathrm O_h$-homogeneous. We have $G=H(1,X_1,\ldots,H_s)\in I$. Furthermore, $H=u^eH_G$ for some $e\geq 0$. Write $H_G=\mathrm{in}(G)+H^1$ where $H^1(0,X_1,\ldots, X_s)=0$. We have $\mathrm{in}(G)=\sum_{i=1}^r\lambda_iF_i$ where $\lambda_i$ is $\mathrm O$-homogeneous of order $p_0-\mathrm O(F_i)$. Let $H_{G^1}=H_G-\sum_{i=1}^r\lambda_i H_{G_i}=H_G-\sum_{i=1}^r\lambda_i H_i$. Then $H_{G^1}\in H_I$ is $\mathrm O_h$-homogeneous and $\mathrm O_h(H_G)<\mathrm O_h(H_{G^1})$. Now we restart with $H_{G^1}$. We prove in this way that $H\in (H_1,\ldots,H_r)$.   
\end{proof} 

Let $H= \sum_{\theta}c_{\theta} u^{\theta}X_1^{\theta_1}\cdots X_s^{\theta_s}\in \ker(H_\psi)$. Write $H= \sum_{k}{H^k}$, where $H^k$ is $\mathrm O_h$-homogeneous. For all $k$, we have $H_\psi(H^k)=0$. Setting $G_k=H^k(1,X_1,\ldots,X_s)$, we have $\psi(G_k)=0$.  This implies that $G_k\in I$. Hence $H_{G_k}\in (H_1,\ldots,H_r)$ by Lemma \ref{standardbasis}. But $H^k=u^{e_k}H_{G_k}$ for some $e_k\in {\mathbb N}$. Consequently $H^k\in (H_1,\ldots,H_r)$. Finally $H\in (H_1,\ldots,H_r)$, which proves that $\ker(H_{\psi})\subseteq J$. As the inclusion $J\subseteq \ker(H_{\psi})$ is obvious, we conclude that $J=\ker(H_{\psi})$.

%=\sum_{\lam:\mathbb{N}^{s+1}\Rigtharrow \mathbb{R}$ be the linear form $H\in \ker(H(\psi))$

%Clearly $J\subseteq \ker(H(\psi))$. Let $H\in \KK\llbracket u,X_1,\ldots,X_s\rrbracket$ and suppose that $H(\psi)(H)=0$. Write

%$$
%H=\sum_{\theta_k, k\geq 0}c_{\theta_k}u^{\theta_k}X_1^{\theta_k^1}\cdots X_s^{\theta_k^s}
%$$

%with  $o(\theta_k)=o_k$ and $o_0+\theta_0\leq o_1+\theta_1\ldots$ Let $H_0=\sum_{\theta_k, o(\theta_k)=o_0+\theta_0}%c_{\theta_k}u^{\theta_k}X_1^{\theta_k^1}\cdots X_s^{\theta_k^s}=u^{p_1}H_0^1+\ldots+u^{p_l}H_0^l$ where $p_1<%%\ldots<p_l$ and $H_0^i(0,X_1,\ldots,X_s)\not=0$ for all $k\in\lbrace 1,\ldots,l\rbrace$. Since $H(\psi)(H)=0$ then  $H_0(u,H_{f_1},%\ldots,H_{f_s})
%=\sum_{\theta_k, o(\theta_k)=o_0+\theta_0}c_{\theta_k}u^{\theta_k}H_{f_1}^{\theta_k^1}\cdots H_{f_s}^{\theta_k^s}t=0$
%If $\theta_0\not=0$, then $H=u^{\theta_0}\bar{H}$ with $H(\phi)(\bar{H})=0$ and if $\bar{H}\in J$ then so is $H$. %Consequentely we shall assume that $\theta_0=0$.
%$\bar{S}_i=F_i+\sum_{\underline{\theta}^i}c^i_{\underline{\theta}^i}X_1^{\theta^i_1}\cdots  X_s^{\theta_s^i}$. Note %that $F_i$ is quasi-homogeneous with respect to the linear form $L(X_i)=\mathrm o(f_i)$. Furthermore, if
%$c^i_{\underline{\theta}^i}\not=0$, then $L(X_1^{\theta^i_1}\cdots  X_s^{\theta_s^i})>L(F_i)$. Let $L(F_i)=D_i$ and let

%$$
%H_L(\bar{S}_i)=F_i+\sum_{\underline{\theta}^i}c^i_{\underline{\theta}^i}u^{L(X_1^{\theta^i_1}\cdots  X_s^{\theta_s^i})- %D_i}X_1^{\theta^i_1}\cdots  X_s^{\theta^i_s}.
%$$

 Now the morphism
%$(H_L(\bar{S}_i))_{1\leq i\leq r}$. Now the morphism 
$$
\KK\llbracket u\rrbracket \longrightarrow {\KK\llbracket u\rrbracket \llbracket X_1,\ldots,X_s\rrbracket }/{J}
$$
is flat because $u$ is not a zero divisor. Hence we get a family of formal space curves  parametrized by $u$ that gives us a deformation from 
${\KK\llbracket X_1,\ldots,X_r\rrbracket }/{I}$ to ${\KK\llbracket X_1,\ldots,X_r\rrbracket }/{(F_1,\ldots,F_r)}$.

In particular we get the following.

\begin{thm} \label{deformation-monomial-local}
Every formal space curve of $\KK^l$, parametrized by $Y_1=g_1(x), \ldots,Y_l=g_l(x)$ has a deformation into a formal monomial curve of $\KK^{r}$ for some positive integer $r$.
\end{thm}

\section{Basis of $\mathbb K\llbracket f(x),g(x)\rrbracket $}\label{two-local}
% In \cite{tor} subalgebras of $K[x]$ the case of finding an algebra basis for $\mathbb K[f(x),g(x)]$ with $f$ and $g$ polynomials is treated. Here we will do the analogue for $\mathbb K[\![f(x),g(x)]\!]$.

In this section we study the particular case of a subalgebra $R$ of $\KK[\![x]\!]$ generated by two elements, and see that a different approach can be considered to study $\mathrm o(R)$, with some interesting applications.
 
Let $f(x)=\sum_{i\geq n}a_ix^i$ and $g(x)=\sum_{j\geq m}b_jx^j$ be two elements of $\KK\llbracket x\rrbracket $ and suppose, without loss of generality, that the following conditions hold:
\begin{enumerate} 
  \item $a_n=b_m=1$,
 \item $n\leq m$,
 \item  the greatest common divisor of $\mathrm{supp}(f(x))\cup\mathrm{supp}(g(x))$ is equal to $1$ (in particular for all $d>1$, $f(x),g(x)\notin\KK\llbracket x^d\rrbracket$). 
 \end{enumerate}
Let the notations be as in Section \ref{formal-space}, in particular ${R}=\KK\llbracket f,g\rrbracket $. By the analytic change of variables $f(x)=\tilde{x}^n$, we may assume that ${R}=\KK\llbracket x^n,g(x)\rrbracket $. Let  $F(X,Y)$ be the $x$-resultant of $X-x^n,Y-g(x)$, that is, $F(X,Y)$ is the generator of the kernel of the map $\rho:\KK\llbracket X,Y\rrbracket \longrightarrow \KK\llbracket x\rrbracket $, $\rho(X)=x^n$ and $\rho(Y)=g(x)$. Since $\KK\llbracket f,g\rrbracket =\KK\llbracket f,g-f^k\rrbracket $ for all $k\geq 1$, then we shall assume that $n< m$ and also that $n$ does not divide $m$. Given a nonzero element $G(X,Y)\notin (F(X,Y))\KK\llbracket X,Y\rrbracket $, we set $\mathrm{int}(F,G)={\mathrm o}(G(f(x),g(x)))$.
 Condition (3) implies that the set of $\mathrm{int}(F,G), G(X,Y)\notin (F(X,Y))\KK\llbracket X,Y\rrbracket $, is a numerical semigroup. We denote it by $\Gamma(F)$. We have the following.
 
 \begin{prop} $\mathrm o({R})=\Gamma(F)$.
 \end{prop}
 
 \begin{proof} We have $a\in \Gamma(F)$ if and only if $a=\mathrm o(G(f(x),g(x)))$ for some $G(X,Y)\in\KK\llbracket X,Y\rrbracket $ if and only if $a\in\mathrm o({R})$. 
 \end{proof}
 
Suppose that $\KK$ is algebraically closed with characteristic zero, and let $d_1=n$, $m_1=\mathrm{inf}\lbrace i\in\mathrm{supp}(g)\mid d_1\nmid i\rbrace$, that is, $m_1=m$, and $d_2=\gcd (n,m_1)$. For all $k\geq 2$ we set $m_k=\mathrm{inf}\lbrace i\in\mathrm{supp}(g) \mid  d_k\nmid i\rbrace$ and $d_{k+1}=\gcd (d_k,m_k)$. It follows that there exists $h\geq 1$ such that $d_{h+1}=1$. The set $\lbrace m_1,\ldots,m_h\rbrace$ is called the \emph{set of Newton-Puiseux exponents} of $F(X,Y)$. Let $e_k=\frac{d_k}{d_{k+1}}$ for all $1\leq k\leq h$ and define the sequence $(r_k)_{0\leq k\leq h}$ as follows: $r_0=n,r_1=m$, and for all $2\leq k\leq h, r_k=r_{k-1}e_{k-1}+m_k-m_{k-1}$. With these notations we have the following (see [1]).
 
 \begin{enumerate}
 
 \item $\Gamma(F)=\mathrm o({R})$ is generated by $\{r_0,r_1,\ldots,r_h\}$.
 
 \item $r_kd_k< r_{k+1}d_{k+1}$ for all $k\in\{1,\ldots,h-1\}$.
 
 \item  $\Gamma(F)=\mathrm o({R})$ is free with respect to the arrangement $(r_0,\ldots,r_h)$. More precisely, let $e_k=\frac{d_k}{d_{k+1}}$ for all $k\in\{1,\ldots,h\}$. Then $e_kr_k\in \langle r_0,\ldots,r_{k-1} \rangle$.
 
 \item $C=\sum_{k=1}^h(e_k-1)r_k -n+1$ is the conductor of $\Gamma(F)=\mathrm o({R})$. 
 \end{enumerate} 
 
 \begin{ex}
 Let $f=x^7$ and $g=x^4+x^2$. The above resultant is then $F=y^7-7x^2y^3-x^4-14x^2y^2-7x^2y-x^2$. Then $\Gamma(F)=\mathrm o(R)=\langle 2,7\rangle$.
\begin{verbatim}
gap> Resultant(x-t^7, y-t^4-t^2,t);
y^7-7*x^2*y^3-x^4-14*x^2*y^2-7*x^2*y-x^2
gap> s:=SemigroupOfValuesOfCurve_Local([t^7,t^4+t^2]);
<Modular numerical semigroup satisfying 7x mod 14 <= x >
gap> MinimalGeneratingSystem(last);
[ 2, 7 ]
gap> IsFreeNumericalSemigroup(s);
true
\end{verbatim}
\end{ex}
 Let the notations be as above. For all $k\geq 2$, let $G_k(X,Y)\in \KK\llbracket X,Y\rrbracket $ such that $\mathrm o(G_k(x^n,g(x)))=r_k$. It follows from \cite{ab} that $\mathrm{deg}_YG_k=\frac{n}{d_k}$. If $g_k(x)= G_k(x^n,g(x))$, then we have the following.
 
 \begin{prop} The set $\lbrace x^n,g,g_2,\ldots,g_h\rbrace$ is a basis of ${R}$, that is, ${R}=\KK\llbracket x^n,g,g_2,\ldots,g_h\rrbracket $ and $\mathrm M_\mathrm o({R})=\KK\llbracket x^n,x^m,x^{r_2},\ldots,x^{r_h}\rrbracket $.
 \end{prop}
 
 %$$
 %N=\mathrm{inf} \lbrace i\in \mathrm{supp}(g)\mid m_1=m<N<m_2, N\in \Gamma(F)\rbrace.
 %$$
 
 %\noindent  Let $a,b\in\NN$ such that $N=an+bm$. If $f_1=x^n, g_1=g-f^ag^m$, then ${R}=\KK\llbracket f_1,g_1\rrbracket $. Let
  %$g_1=x^m+\sum_{j>m}b^1_jx^j$. It follows that $S_1= \lbrace i\in \mathrm{supp}(g_1)\mid m_1=m<N<m_2, N\in %\Gamma(F)\rbrace$ is not empty, then inf$(S_1)=N_1>N$. In this case  we restart with $N_1$... The set of integers %$m<i<m_2$ being finite, we may assume that $f=x^n, g=x^m+\sum_{j\langle m_1, j\notin \Gamma(F)}b_{j}x^j+\sum_{j \ranglem_2}b_jx^j$. %Let 
  
  %$$
  %N'=\mathrm{inf} \lbrace i\in \mathrm{supp}(g)\mid m_2<N<m_3, N\in \Gamma(F)\rbrace.
  %$$
  
  %\noindent The $N'=a'n+b'm+c'r_2$. We set  $f'=x^n, g'=g-f^{a'}g^{m'}G_k(f,g)^{c'}$. Then $A =\KK\llbracket f',g'\rrbracket $ and
  %$f'=x^n, g'=x^m+\sum_{j\langle m_1, j\notin \Gamma(F)}b_{j}x^j+\sum_{j \ranglem_2}b'_jx^j$ and if  $S'= \lbrace i\in {\rm %Supp}(g')\mid m_2<N<m_3, N\in \Gamma(F)\rbrace$ is not empty, then inf$(S')>N'$....
  
  %\noindent Repeating this argument until $m_h$, and then over the set of elements $N\geq C$, we may assume that
  
  %$$
  %f=x^n, g=\sum_{j<C, j\notin\Gamma(F)}b_jx^j.
  %$$ 
 
 \noindent Note that, by a similar argument as in Section \ref{formal-space}, we may assume that  $f=x^n, g=x^m+\sum_{i\in \mathrm G(\Gamma(F))}c^1_ix^i$, and for all $k\geq 2, g_k=x^{r_k}+\sum_{i\in \mathrm G(\Gamma(F))}c^k_ix^i$, where $\mathrm G(\Gamma(F))=\lbrace j\in \mathbb N\mid  j\notin \Gamma(F)\rbrace$ is the set of gaps of $\Gamma(F)$.
 
 % then we may assume that 
 %$f=x^n, g=\sum_{j\in G(\Gamma(F))}b_jx^j$. Note that, by construction, $\lbrace m_2,\ldots,m_h\rbrace\subseteq %G(\Gamma(F))$.
 
 \noindent Let the notations be as in Section \ref{deformation-local}. The morphism
 $$
 D:\KK\llbracket u\rrbracket \longrightarrow T= \KK\llbracket u\rrbracket \llbracket H_f,H_g,H_{g_2},\ldots,H_{g_h}\rrbracket 
 $$
 
 \noindent gives us a deformation of $T\mid_{u=1}={R}=\KK\llbracket f(x),g(x),g_2(x),\ldots,g_h(x)\rrbracket $ to $T\mid_{u=0}=\KK\llbracket x^n,x^m,x^{r_2},\ldots,x^{r_h}\rrbracket $. Note that,
 since $\langle n,m,r_2,\ldots,r_h \rangle$ is free with respect to the given arrangement, then it is a complete intersection (see for instance \cite{ns}). For all $k\in\{1,\ldots,h\}$, write $e_kr_k=\sum_{i=0}^{k-1}\theta^k_ir_i$ with $0\leq \theta^k_i<e_i$ for all $i\in\{1,\ldots,k-1\}$. If $B$ is the ideal of $\KK[X_0,X_1,\ldots,X_h]$ generated by
 $$
 \lbrace X_1^{e_2 }-X_0^{m\over d_2}, X_2^{e_2}-X_0^{\theta^2_0}X_1^{\theta^2_1},\ldots, X_{h}^{e_h}-X_0^{\theta^{h}_0}X_1^{\theta^{h}_1}\ldots X_{h-1}^{\theta^{h}_{h-1}}\rbrace
 $$
 then 
 $$
 \KK\llbracket x^n,x^m,x^{r_2},\ldots,x^{r_h}\rrbracket \simeq {\KK\llbracket X_0,X_1,\ldots,X_h\rrbracket }/{B}.
 $$
 Let $\bar{F}(X,Y)$ be the $x$-resultant of $X-x^n, y-g(x)$. By hypothesis,  $\bar{F}(X,Y)$ is a polynomial. Furthermore, $\bar{F}(X,Y)=Y^n+a_1(X)Y^{n-1}+\ldots+a_n(X)$ with $\mathrm o(a_i(X))>i$ for all $2\leq i\leq n$. Set $G_{h+1}={\bar F}$ and for all $k\geq 1$, let
  $$
 G_{k+1}=G_k^{e_k}-X^{\theta^k_0}\prod_{i=1}^{k-1}G_i^{\theta^k_i}+
 \sum_{\underline{\alpha}^k}c^k_{\underline{\alpha}^k}X^{\alpha^k_0}G_1^{\alpha^k_1}\cdots G_{k}^{\alpha^k_{k}},
 $$
 where the following conditions hold:
 
 \begin{enumerate}
 
 \item for all $i\in\{1,\ldots,k-1\}$, $0\leq \theta^k_i<e_i$;
 
 \item for all $\underline{\alpha}^k$, if $c^k_{\underline{\alpha}^k}\not= 0$, then for all $i\in\{1,\ldots,k\}$, $0\leq \alpha^k_i<e_i$;
 
 \item for all $\underline{\alpha}^k$, if $c^k_{\underline{\alpha}^k}\not= 0$, then $\alpha^k_0n+\sum_{i=1}^{k}\alpha^k_ir_i=D^k_i>e_kr_k=\theta^k_0r_0+\sum_{i=1}^{k-1}\theta^k_ir_i$.
 \end{enumerate}
It follows from Section \ref{deformation-local}. that if $I$ (respectively $J$) is the ideal generated by \[(X_k^{e_k}-X_0^{\theta^k_0}\prod_{i=1}^{k-1}X_i^{\theta^k_i}+\sum_{\underline{\alpha}^k}c^k_{\underline{\alpha}^k}X_0^{\alpha^k_0}X_1^{\alpha^k_1}\cdots  X_{k}^{\alpha^k_{k}})_{1\leq k\leq h}\] (respectively $(X_k^{e_k}-X_0^{\theta^k_0}\prod_{i=1}^{k-1}X_i^{\theta^k_i}+\sum_{\underline{\alpha}^k}c^k_{\underline{\alpha}^k}u^{D^k_i-e_kr_k}X_0^{\alpha^k_0}X_1^{\alpha^k_1}\cdots  X_{k}^{\alpha^k_{k}})_{1\leq k\leq h}$) in $\KK\llbracket X_0,\ldots,X_h\rrbracket $ (respectively $\KK\llbracket u\rrbracket \llbracket X_0,\ldots,X_h\rrbracket $), then 
  $$
    {R}=\KK\llbracket x^n,g(x),g_2(x),\ldots,g_h(x)\rrbracket \simeq {\KK\llbracket X_0,X_1,\ldots,X_h\rrbracket }/{I}
    $$
  and 
 $$
     \KK\llbracket u\rrbracket \llbracket x^n,H_g,H_{g_2},\ldots,H_{g_h}\rrbracket\simeq {\KK\llbracket u\rrbracket \llbracket X_0,X_1,\ldots,X_h\rrbracket }/{J}.
     $$
	Furthermore, ${\KK\llbracket u\rrbracket \llbracket X_0,X_1,\ldots,X_h\rrbracket }/{J}$ is a flat $\KK\llbracket u\rrbracket$-module. This gives us a family of formal space curves parametrized by $u$ which is a deformation from ${\KK\llbracket X_0,X_1,\ldots,X_h\rrbracket }/{I}$ to 
 the formal toric variety ${\KK\llbracket X_0,X_1,\ldots,X_h\rrbracket }/{B}$. The later being a complete intersection, we get the following.

 %$$
  %G_{k+1}=G_k^{e_k}-X_0^{\theta^k_0}X_1^{\theta^k_1}\cdots  X_{k-1}^{\theta^k_{k-1}}+\sum_{\underline{\alpha}^k}c^k_{\un %derline{\alpha}^k}X^{\alpha^k_0}.G_1^{\alpha^k_1}\cdots  G_{k-1}^{\alpha^k_{k-1}}
  %$$
  
%\noindent Where for all $\underline{\alpha}^k$, if $c^k_{\underline{\alpha}^k}\not= 0$, then %$\sum_{i=1}^{k-1}\theta^k_ir_i<e_kr_k$. If $J$ denotes the ideal generated by %$(X_k^{e_k}-X_0^{\theta^k_0}X_1^{\theta^k_1}\cdots  X_{k-1}^{\theta^k_{k-1}}+\sum_{\underline{\alpha}^k}c^k_{\underline{ %\alpha}^k}X_0^{\alpha^k_0}.X_1^{\alpha^k_1}\cdots  X_{k-1}^{\alpha^k_{k-1}})_{1\leq k\leq h}$, then 

 %$$
  %{R}=\KK\llbracket x^n,g(x),g_2(x),\ldots,g_h(x)]\simeq \frac{\KK\llbracket X_0,X_1,\ldots,X_h\rrbracket }{I}
  %$$
 
 %\noindent and $\frac{\KK\llbracket X_0,X_1,\ldots,X_h\rrbracket }{H(J)}$
 \begin{thm} Every irreducible singularity of a plane curve $X=f(x), Y=g(x)$ of $\KK^2$ has a deformation into a formal monomial complete intersection curve of $\KK^{h+1}$ for some $h\geq 1$.
 \end{thm}
 
 \begin{ex}
 Let $f(x)=x^4,  g(x)=x^6+x^7$. The minimal polynomial of $(f(x),g(x))$ is given by:
$$
F(X,Y)=Y^4-2X^3Y^2+X^6-4X^5Y-X^7=(Y^2-X^3)^2-4X^5Y-X^7
$$
Let $r_0=4=d_1,r_1=6=m_1$ and $G_1=Y$. We have $d_2=\gcd (6,4)=2$, hence $m_2=7$. It follows that $r_2=13$. Note that if $G_2=Y^2-X^3$, then  $g_2(x)=G_2(f(x),g(x))=2x^{13}+x^{14}$. Hence $\Gamma(F)=\mathrm o({R})=\langle 4,6,13 \rangle$ and $\lbrace f(x),g(x),g_2(x)\rbrace$ is a basis of ${R}$. Let us double check it.
\begin{verbatim}
gap> SemigroupOfValuesOfCurve_Local([x^4,x^6+x^7],"basis");
[ x^4, x^7+x^6, -1/2*x^15+x^13 ]
\end{verbatim} 
(Observe that the output is different, since this is a reduced basis: we change $2x^{13}+x^4$ with $x^{13}+\frac{1}2x^{14}$, and then using that $14=2\times 4+6$, we replace this last polynomial with $x^{13}-\frac{1}2x^{15}$.)

Consequently, $H_{R}=\KK\llbracket u,x^4,x^6+ux^7,2x^{13}+ux^{14}\rrbracket$. With the notations above, $e_1=3, e_2=2$, hence $\KK\llbracket x^4,x^6,x^{13}\rrbracket \simeq T={\KK\llbracket X_0,X_1,X_2\rrbracket}/{(X_1^2-X_0^3, X_2^2-X_0^5X_1)}$, and
$$
\KK\llbracket u\rrbracket \longrightarrow {\KK\llbracket u\rrbracket \llbracket X_0,X_1,X_2\rrbracket }/{(X_1^2-X_0^3, X_2^2-4X_0^5X_1-u^{2}X_0^7)}
$$
gives us a deformation from ${R}$ to $T$ (we can also change $X_2$ with $\frac{1}2X_2$, and then $B=(X_1^2-X_0^3,X_2^2-4X_0^4X_1)$).
\end{ex}

\section{Semigroup of a polynomial curve}\label{polynomial-curve}

Let $\KK$ be a field and let $f_1(x),\ldots,f_s(x)$ be $s$ polynomials of $\KK[x]$. Let $A =\KK[f_1,\ldots,f_s]$ be a subalgebra of $\KK[ x]$,  and assume, without loss of generality, that $f_i$ is monic for all $i\in\{1,\ldots,s\}$. Given $f(x)=\sum_{i=0}^pc_ix^i\in A $, with $c_p\neq 0$, we set $\mathrm d(f)=p$ and $\mathrm M(f)=c_px^p$, the \emph{degree} and \emph{leading monomial}, respectively. We also define $\mathrm{supp}(f)=\lbrace i \mid c_i\not=0\rbrace$. The set $\mathrm d(A)=\lbrace \mathrm d(f) \mid f\in  A \rbrace$ is a submonoid of $\NN$. We shall  assume that $\lambda_A ({\KK[x]}/A )<\infty$. In particular $\mathrm d(A)$ is a numerical semigroup. %We denote it by $\mathrm d(A)$. 
We say that $\lbrace f_1,\ldots,f_s\rbrace$ is a \emph{basis} of $A$ if $\{\mathrm d(f_1),\ldots,\mathrm d(f_s)\}$ generates $\mathrm d(A)$. Clearly, $\lbrace f_1,\ldots,f_s\rbrace$ is a basis of $A$ if and only if $\KK[M(f),f\in A ]=\KK[M(f_1),\ldots,M(f_s)]$. For several variables, these basis are known in the literature as SAGBI basis (\cite{rs,br}). Since there are already algorithms in the literature to calculate a basis of $A$, we will not include the procedure here. 

We would like just mention that if we follow a similar argument to the one used in Section \ref{formal-space}, the sequences of degrees decrease, and thus the finiteness conditions are easier to deduce. In this setting a basis for $A$ is unique up to constants.

\section{Deformation to a toric ideal}\label{deformation-global}
 
Let the notations be as in Section \ref{polynomial-curve}. Given $f(x)=\sum_{i=0}^pc_ix^i\in\KK[x]$, we set $h_f(u,x)=\sum_{i=0}^pc_iu^{p-i}x^i$, in particular, if we consider the linear form $L_h:{\mathbb N}^2\longmapsto {\mathbb N}, L(a,b)=a+b$, then $h_f$ is $L_h$-homogeneous of degree $p$, that is, $L_h(i,p-i)=p$ for all $i\in{\rm Supp}(f)$ . We set $h_A =\KK[u,h_f \mid f\in A ]$. With these notations we have the following result, and its proof is similar to that of Proposition \ref{homogeneous-local}.

\begin{prop}
The set $\lbrace f_1,\ldots,f_r\rbrace$ is a basis of $A$ if and only if $h_A =\KK[u,h_{f_1},\ldots,h_{f_s}]$.
\end{prop}

Suppose that $\lbrace f_1,\ldots,f_s\rbrace$ is a basis of $A$. By the inclusion morphism of rings  $D:\KK[u]\longrightarrow B=\KK[u,h_{f_1},\ldots,h_{f_s}]$, $B$ is a $\KK[u]$-module. When $u=1$ (respectively $u=0$), we get $B\mid_{u=1}=A$ (respectively $B\mid_{u=0}=\KK[\mathrm M(f_1),\ldots,\mathrm M(f_s)]$). Hence we get a deformation from $A$ to $\KK[\mathrm M(f_1),\ldots,\mathrm M(f_s)]$. More precisely let 
$$
\psi:\KK[X_1,\ldots,X_s]\longrightarrow \KK[{f_1},\ldots,{f_s}]
$$
and 
$$
h_{\psi}:\KK[u][X_1,\ldots,X_s]\longrightarrow \KK[u][h_{f_1},\ldots,h_{f_s}]
$$
be the morphisms of rings such that $h_{\psi}(u)=u$, $\psi(X_i)=f_i$ and $h_{\psi}(X_i)=h_{f_i}$ for all $i=1,\ldots,s$. For all $i=1,\ldots,r$, let 
$$
S_i=f_1^{\alpha^i_1}\cdots  f_s^{\alpha^i_s}-f_1^{\beta^i_1}\cdots  f_s^{\beta^i_s}=\sum_{\underline{\theta}^i}c^i_{\underline{\theta}^i}f_1^{\theta^i_1}\cdots  f_s^{\theta_s^i}
$$ 
with $\mathrm d(f_1^{\theta^i_1}\cdots  f_s^{\theta_s^i})= D^i_{\underline{\theta}^i}>\sum_{k=1}^s\alpha^i_k\mathrm d(f_k)=\sum_{k=1}^s\beta^i_k\mathrm d(f_k)=p_i$. Let $I$ (respectively $J$) be the ideal generated by $(G_i=X_1^{\alpha^i_1}\cdots  X_s^{\alpha^i_s}-X_1^{\beta^i_1}\cdots  X_s^{\beta^i_s}-\sum_{\underline{\theta}^i}c^i_{\underline{\theta}^i}X_1^{\theta^i_1}\cdots  X_s^{\theta_s^i})_{1\leq i\leq r}$ (respectively $(H_i=X_1^{\alpha^i_1}\cdots  X_s^{\alpha^i_s}-X_1^{\beta^i_1}\cdots  X_s^{\beta^i_s}-\sum_{\underline{\theta}^i}u^{p_i-D^i_{\underline{\theta}^i}}c^i_{\underline{\theta}^i}X_1^{\theta^i_1}\cdots  X_s^{\theta_s^i})_{1\leq i\leq r}$) in $\KK[X_1,\ldots,X_s]$ (respectively $\KK[u][X_1,\ldots,X_s]$). 

Well shall consider on ${\mathbb N}^s$ (respectively, $\mathbb{N}^{s+1}$) the linear form 
\[\mathrm D(\theta_1,\ldots,\theta_s)=\sum_{i=1}^s\theta_i\mathrm d(f_i)\] 
(respectively $\mathrm D_h(\theta_0,\theta_1,\ldots,\theta_s)=\theta_0+\sum_{i=1}^s\theta_i \mathrm d(f_i)$).

Given a monomial $X_1^{\theta_1}\cdots X_s^{\theta_s}$ (respectively $u^{\theta_0} X_1^{\theta_1}\cdots X_s^{\theta_s}$), we set $\mathrm D(X_1^{\theta_1}\cdots X_s^{\theta_s})=\mathrm D(\theta_1,\ldots,\theta_s)$ (respectively $\mathrm D_h(u^{\theta_0} X_1^{\theta_1}\cdots X_s^{\theta_s})=\mathrm D_h(\theta_0,\theta_1,\ldots,\theta_s)$). 

For $G=\sum_{\theta} c_{\theta}X_1^{\theta_1}\cdots X_s^{\theta_s}$ (respectively $H=\sum_{\theta} c_{\theta}u^{\theta_0}X_1^{\theta_1}\cdots X_s^{\theta_s}$), we say that $G$ (respectively $H$) is $\mathrm D$-homogeneous of degree $a$ (respectively $\mathrm D_h$-homogeneous of degree $b$) if  $\mathrm D(\theta_1,\ldots,\theta_s)=a$ (respectively $\mathrm D_h(\theta_0,\theta_1,\ldots,\theta_s)=b$) for all $(\theta_1,\ldots,\theta_s)$ (respectively $(\theta_0,\theta_1,\ldots,\theta_s)$) such that $c_{\theta}\not=0$. More generally let $G=\sum_{k=0}^mG_{p_k}$  where $p_0>p_1>\cdots>p_m$  and $G_{p_k}$  is $\mathrm D$-homogeneous of degree $p_k$. We set $\mathrm D(G)=p_0$. We also set $\mathrm{In}(G)=G_{p_0}$ and we call it the \emph{initial form} of $G$. 

%We finally set $\mathrm O(0)=+\infty$, and we recall that $\mathrm O(G)=+\infty$ if and only if $G=0$.
  %$\bar{S}_i=F_i+\sum_{\underline{\theta}^i}c^i_{\underline{\theta}^i}X_1^{\theta^i_1}\cdots  X_s^{\theta_s^i}$. Note %that $F_i$ is quasi-homogeneous with respect to the linear form $L(X_i)=\mathrm o(f_i)$. Furthermore, if
  %$c^i_{\underline{\theta}^i}\not=0$, then $L(X_1^{\theta^i_1}\cdots  X_s^{\theta_s^i})>L(F_i)$. Let $L(F_i)=D_i$ and let
  
  %$$
  %H_L(\bar{S}_i)=F_i+\sum_{\underline{\theta}^i}c^i_{\underline{\theta}^i}u^{L(X_1^{\theta^i_1}\cdots  X_s^{\theta_s^i})- %D_i}X_1^{\theta^i_1}\cdots  X_s^{\theta^i_s}.
  %$$
 \begin{lem}\label{kerglobal} With the standing notations and hypothesis, the kernel of $\psi$ is generated by $I$. 
 \end{lem}
 
 \begin{proof} Let  $F_1,\ldots, F_r$ be a generating system of the kernel of the morphism  \[\phi: \KK[ X_1,\ldots,X_s] \longrightarrow \KK[x],\ \phi(X_i)=\mathrm M(f_i)\]
  for all $i\in \lbrace 1,\ldots,s\rbrace$. In particular $F_i$ is $\mathrm D$-homogeneous of degree $\mathrm d(f_i)$ for all $i\in\lbrace 1,\ldots, s\rbrace$, and ${\KK[ X_1,\ldots,X_s]}/{(F_1,\ldots,F_r)}\simeq \KK[ \mathrm M(f_1),\ldots,\mathrm  M(f_s)]$. 
 
 For all $i\in\lbrace 1,\ldots, r\rbrace, \psi(G_i)=0$. Hence $I\subseteq \ker(\psi)$. For the other inclusion, let $G=\sum_{\theta}c_{\theta}X_1^{\theta_1}\cdots X_s^{\theta_s}\in \ker (\psi)$. Write $G=\sum_{k=0}^mc_{\theta^k}X_1^{\theta^k_1}\cdots X_s^{\theta^k_s}$ where $\mathrm D(\theta^0)\geq\mathrm  D(\theta^{1})\geq\cdots >\mathrm  D(\theta^m)$. Since $\psi(G)=0$, we have that 
 \[
 \sum_{k=0}^m c_{\theta^k}f_1^{\theta^k_1}\cdots f_s^{\theta^k_s}=0.
 \]
 In particular, $\sum_{k, D(\theta^k)=D(\theta^0)}c_{\theta^k}\mathrm M(f_1)^{\theta^k_1}\cdots \mathrm  M(f_s)^{\theta^k_s}=0$, and consequently $\sum_{k, D(\theta^k)=D(\theta^0)}c_{\theta^k}X_1^{\theta^k_1}\cdots X_s^{\theta^k_s}\in \ker(\phi)$. This implies that 
 \[
 \sum_{k, D(\theta^k)=D(\theta^0)}c_{\theta^k}X_1^{\theta^k_1}\cdots X_s^{\theta^k_s}=\sum_{i=1}^r\lambda^0_iF_i
 \]
 for some $\lambda^0_i\in \KK[ X_1,\ldots,X_s]$, $i\in\lbrace 1,\ldots,r\rbrace$, with $\lambda^0_i$ is $\mathrm D$-homogeneous of degree $\mathrm D(G)-\mathrm D(F_i)$. Hence
 \[
 \sum_{k, D(\theta^k)=D(\theta^0)}c_{\theta^k}f_1^{\theta^k_1}\cdots f_s^{\theta^k_s}=\sum_{i=1}^r\lambda^0_i(f_1,\ldots,f_s)S_i.
 \]
 Let $G^1=G-\sum_{i=1}^r\lambda^0_iG_i$. It follows that $G^1\in \ker(\psi)$. If $G^1\not=0$, then $\mathrm D(G)>\mathrm D(G^1)$. Then we restart with $G^1$. We construct  in the same way 
 $G^2$, and $\lambda^1_1,\ldots,\lambda^1_r$ such that $G^1=G^2+\sum_{j=1}^r\lambda^1_jG_j$ with $\mathrm D(G)>\mathrm D(G^1)>\mathrm  D(G^2)$, $\lambda^1_i$ $\mathrm D$-homogeneous and  $\mathrm D(\lambda^0_i)>\mathrm D(\lambda^1_i)$ for all $i\in\lbrace 1,\ldots,r\rbrace$. If we continue in this way, we get that for all $k\geq 0$, 
 \[
 G=G^{k+1}+\sum_{i=1}^r(\lambda^0_i+\lambda^1_i+\ldots+\lambda^k_i)G_i,
 \]
 with $\mathrm D(G)>\mathrm D(G^1)>\cdots>\mathrm D(G^{k+1})$, $\lambda_i^j$  $D$-homogeneous, and $\mathrm D(\lambda^0_i)>\mathrm D(\lambda^1_i)>\cdots > \mathrm D(\lambda^k_i)$ for all $i\in\{ 1,\ldots,r\}$ and for all $j\in\lbrace 1,\ldots,k\rbrace$. Thus, there exists $l$ such that  $G^{l+1}=0$. Hence  $G=\sum_{i=1}^r(\lambda^0_i+\lambda^1_i+\ldots+\lambda^l_i)G_i$. This proves our assertion.
 \end{proof}
 
  Let the notations be as above. Let $G= \sum_{\theta}c_{\theta} X_1^{\theta_1}\cdots X_s^{\theta_s} \in \KK[ X_1,\ldots,X_s$ and write  $G=\sum_{i= 0}^m{G_{p_i}}$ with $p_0>p_1>\cdots>p_m$ and $G_{p_i}$  $\mathrm D$-homogeneous. We set $H_G=\sum_{i= 0}^mu^{p_0-p_i}G_{p_i}$, in such a way that $H_G$ is $D_h$-homogeneous of degree $p_0$. Given an ideal $S$ of $\KK[ X_1,\ldots,X_s]$, we set $\mathrm{In}(S)=(\mathrm{In}(G) \mid G\in S\setminus\{0\})$. We also denote by $H_S=(H_G \mid G\in S\setminus\{0\})\KK[u,X_1,\ldots,X_s]$. With these notations we have $\mathrm{In}(S_i)=F_i$ and $H_{G_i}=H_i$ for all $i\in\lbrace 1,\ldots,r\rbrace$.
 
 \begin{lem}\label{grobnerbasis} Let the notations be as above. We have $\mathrm{In}(I)=(F_1,\ldots,F_r)$, and 
 $H_I=(H_1,\ldots,H_r)=J$.
 \end{lem}
 
 \begin{proof} The first assertion follows from the proof of  Lemma \ref{kerglobal}. 
 
 To prove the second assertion, let
 $H\in H_I$ and assume that $H$ is $D_h$-homogeneous. We have $G=H(1,X_1,\ldots,H_s)\in I$. Furthermore, $H=u^eH_G$ for some $e\geq 0$. Write $H_G={\rm In}(G)+H^1$ where $H^1(0,X_1,\ldots, X_s)=0$. We have $\mathrm{In}(G)=\sum_{i=1}^r\lambda_iF_i$ where $\lambda_i$ is $D$-homogeneous of degree $D(G)-D(F_i)$. Let $H_{G^1}=H_G-\sum_{i=1}^r\lambda_i H_{G_i}=H_G-\sum_{i=1}^r\lambda_i H_i$. Then $H_{G^1}\in H_I$ is $D_h$-homogeneous and $D_h(H_G)>D_h(H_{G^1})$. Now we restart with $H_{G^1}$. In this way we show that $H\in (H_1,\ldots,H_r)$.   
 \end{proof} 
 
Let $H=\sum_{\theta}c_{\theta}u^{\theta}X_1^{\theta_1}\cdots X_s^{\theta_s}\in \ker(h_\psi)$. Write $H=\sum_{k= 0}^n{H^k}$ where $H^k$ is $D_h$-homogeneous. For all $k$, we have $h_\psi(H^k)=0$. Setting $G_k=H^k(1,X_1,\ldots,X_s)$, we have  $\psi(G_k)=0$.  This implies that $G_k\in I$, and thus $H_{G_k}\in (H_1,\ldots,H_r)$ by Proposition \ref{grobnerbasis}. But $H^k=u^{e_k}H_{G_k}$ for some $e_k\in {\mathbb N}$, whence $H^k\in (H_1,\ldots,H_r)$. Finally $H\in (H_1,\ldots,H_r)$, which proves that $\ker(h_\psi)\subseteq J$. The inclusion $J\subseteq \ker(h_\psi)$ is obvious, and we can conclude that $J=\ker(h_\psi)$.
  
 %Since $\lbrace f_1,\ldots,f_s\rbrace$ is a basis of $A$, then the kernel of $\psi$ (respectively $h_{\psi}$) is generated %by $I$ (respectively $J$).
 
 Now the morphism
%$(H_L(\bar{S}_i))_{1\leq i\leq r}$. Now the morphism 
$$
\KK[u]\longrightarrow {\KK[u][X_1,\ldots,X_r]}/{J}
$$
is flat (because $p(u)$ is not a zero divisor for all $p(u)\in \KK[u]$). Hence we get a family of polynomial space curves  parametrized by $u$ which gives us a deformation from 
${\KK[X_1,\ldots,X_r]}/{I}$ to ${\KK[X_1,\ldots,X_r]}/{(F_1,\ldots,F_r)}$.
%\end{rem}
  
In particular, we get the following analogue to Theorem \ref{deformation-monomial-local}, which can be seen as a a geometric reinterpretation of \cite[Corollary 11.6]{st} (also \cite[Corollary 6.1]{br}).
  
\begin{thm} 
Every polynomial space curve of $\KK^l$, parametrized by $Y_1=g_1(x), \ldots,Y_l=g_l(x)$ has a deformation into a  monomial curve of $\KK^{r}$ for some positive integer $r$.
\end{thm}

\section{Basis of $K[f(x),g(x)]$}

In \cite{tor}, the case when a subalgebra $A$ of $\KK[x]$ has a (SAGBI) basis with two elements is treated. Here we study subalgebras generated by two elements of $\KK[x]$, and see how a basis can be obtained by using a different approach to that of the general case, as we already did for $\KK[\![x]\!]$ in Section \ref{two-local}.

Let $f(x)=\sum_{i=1}^na_ix^i$ and $g(x)=\sum_{j=1}^mb_jx^j$ be two polynomials of $\KK[x]$ and suppose, without loss of generality, that the following conditions hold:
\begin{enumerate} 
\item $a_n=b_m=1$,
\item $n\geq m$,
\item the greatest common divisor of $\mathrm{supp}(f(x))\cup\mathrm{supp}(g(x))$ is equal to $1$ (in particular for all $d>1$, $f(x),g(x)\notin\KK [x^d]$).
\end{enumerate}
Let the notations be as in Section \ref{polynomial-curve}, in particular $A =\KK[f,g]$. Let also $F(X,Y)$ be the $x$-resultant of $X-f(x),Y-g(x)$, that is, $F(X,Y)$ is the generator of the kernel of the map $\psi:\KK[X,Y]\longrightarrow \KK[x], \psi(X)=f(x)$ and $\psi(Y)=g(x)$. Since $\KK[f,g]=\KK[f,g-f]$, then we shall assume that $n> m$. Write $F(X,Y)=Y^n+c_1(X)Y^{n-1}+\dots+c_n(X)$. Given a polynomial $G(X,Y)\notin (F(X,Y))\KK[X,Y]$, we set $\mathrm{int}(F,G)=\mathrm{deg}_xG(f(x),g(x))$. Assume that $\KK$ is algebraically closed with characteristic zero. Let $d$ be a divisor of $n$, and let $G$ be a monic polynomial in $\KK[X][Y]$ of degree $\frac{n}{d}$ in $Y$. Write $F=G^d+\alpha_1(X,Y)G^{d-1}+\dots+\alpha_d(X,Y)$ where for all $k\in \{1,\ldots,d\}$, if $\alpha_k\not=0$, then $\mathrm{deg}_Y\alpha_k<\frac{n}{d}$. We say that $G$ is a $d$th \emph{approximate root} of $F$ if $\alpha_1=0$. There is a unique $d$th approximate root of $F$. We denote it by App$(F,d)$. The following results can be found in \cite{ab}.

\begin{thm} \label{polynomialisoneplace}Under the standing hypothesis.
\begin{enumerate} \item  $F(X,Y)$ has one place at infinity, that is the affine curve $F(X,Y)=0$ has one point at infinity, and the projective closure of this curve
in $\mathbb{P}^2_{\mathbb K}$ is analytically irreducible at this point.

%$F(X,Y)$ has one place at infinity, that is, for all $k\in\{1,\ldots,n\}$, if $c_k(X)\not=0$, then  $\mathrm{deg}_Xc_k(X)< k$.

\item $\lbrace \mathrm{int}(F,G)\mid  G\in\KK[X,Y]\setminus (F)\rbrace$ is a numerical semigroup. 

\item Let $D(n)$ be the set of divisors of $n$. The set $\lbrace \mathrm{int}(F,\mathrm{App}(F,d))\mid d\in D(n)\rbrace$ generates $\Gamma(F)$. 
\end{enumerate}
\end{thm}

We call $\lbrace \mathrm{int}(F,G)\mid  G\in\KK[X,Y]\setminus (F)\rbrace$ the semigroup of $F$, and we denote it by $\Gamma(F)$.

\begin{cor} 
Let the notations be as above. We have $\mathrm d(A)=\Gamma(F)$.
\end{cor}

\begin{proof} In fact, $h(x)\in A $ if and only if $h(x)=P(f(x),g(x))$ for some $P(X,Y)\in\KK[X,Y]$. Hence $a\in\mathrm d(A)$ if and only if $a=\mathrm{int}(F,P),P\in\KK[X,Y]$ which means that $a\in\Gamma(F)$.
\end{proof}

\noindent Let $F(X,Y)=Y^n+c_1(X)Y^{n-1}+\dots+c_n(X)$ be as above, and assume, after a possible change of variables $X'=X,Y'=Y+\frac{c_1}{n}$, that $c_1(X,Y)=0$ (note that this does not change $A $). In particular App$(F,n)=Y$. A system of generators of $\Gamma(F)$ can be found algorithmically in the following way.

Let $r_0=d_1=n=\mathrm{int}(F,X), r_1=\mathrm{deg}_Xa_n(X)=\mathrm{int}(F,\mathrm{App}(F,n))$, and $d_2=\gcd (r_0,r_1)$. We set 
$G_2=\mathrm{App}(F,d_2), r_2=\mathrm{int}(F,G_2)=\mathrm{deg}_xG_2(f(x),g(x))$, and $d_3=\gcd (r_3,d_2)$, and so on\dots With these notations we have the following:
\begin{enumerate}
\item $d_1>d_2>\ldots$ and there exists $h\geq 1$ such that $d_{h+1}=1$;
\item $\Gamma(F)=\mathrm d(A)$ is generated by $\{r_0,r_1,\ldots,r_h\}$;
\item $r_kd_k>r_{k-1}d_{k-1}$ for all $k\in \{1,\ldots,h\}$;
\item  $\Gamma(F)=\mathrm d(A)$ is free with respect to the arrangement $(r_0,\ldots,r_h)$. More precisely, let $e_k=\frac{d_k}{d_{k+1}}$ for all $k\in\{1,\ldots,h\}$. Then $e_kr_k\in \langle r_0,\ldots,r_{k-1} \rangle$;
\item $C=\sum_{k=1}^h(e_k-1)r_k -n+1$ is the conductor of $\Gamma(F)=\mathrm d(A)$. 
\end{enumerate} 

\begin{lem}
If $A =\KK[x]$, then $r_k=d_{k+1}$ for all $k=1,\ldots,h$. In particular $\mathrm{deg}_xG_h(f(x),g(x))=1$ and $m$ divides $n$.
\end{lem}
\begin{proof} 
If $A =\KK[x]$ then $C=0$, hence $\sum_{k=1}^h(e_k-1)r_k =n-1$. Since $r_{k}\geq d_{k+1}$, then 
$\sum_{k=1}^h(e_k-1)r_k\geq n-1$ with equality if and only if $r_k=d_{k+1}$ for all $k=1,\ldots,h$. Since $m=r_1=d_2=\gcd (n,m)$, then $m$ divides $n$.
\end{proof}

\begin{lem}\cite[Theorem 2]{tor}
If $\gcd (n,m)=1$, then $\lbrace f(x),g(x)\rbrace$ is a basis of $A $. 
\end{lem}
\begin{proof} If $\gcd (n,m)=1$, then $\Gamma(F)=\mathrm d(A)=\langle n,m \rangle$. Hence  $\lbrace f(x),g(x)\rbrace$ is a basis of $A $.
\end{proof}

\begin{lem}
Suppose that $\gcd (n,m)=p_1\cdots  p_l$ where $p_i$ is a positive prime number for all $i\in\{1,\ldots,l\}$ (and the $p_i$'s are not necessarily distinct). The set $\lbrace f(x),g(x)\rbrace$ is not a basis of $A $. Furthermore, if $c$ is the cardinality  of a basis of  $A $, then $2\leq c\leq l+2$. In particular, if $\gcd (n,m)$ is a prime number $p>1$, then a basis of  $A $ has either two or three elements.
\end{lem}
\begin{proof} 
Since $\gcd (n,m)>1$, then the first assertion is clear. On the other hand, since $d_2=\gcd (n,m)=p_1\cdots  p_r$, we have $A \not=\KK[x]$, and $h\leq l+1$. Hence  $\Gamma(F)=\mathrm d(A)$ has at most $l+2$ generators. The result now follows.
\end{proof}

\begin{rem}
Let $\underline{r}=(r_0=n, r_1=m, r_2,\ldots,r_h)$ be a sequence of integers and for all $k\geq 1$, let $d_k=\gcd(r_0,\cdots,r_{k-1})$ and $e_k=\frac{d_k}{d_{k+1}}$. Assume that the following conditions hold:
\begin{enumerate}
\item $d_1>d_2>\ldots>d_{h+1}=1$;
\item $r_kd_k>r_{k-1}d_{k-1}$ for all $k\in\{1,\ldots,h\}$;
\item $e_kr_k\in <r_0,\ldots,r_{k-1}>$ for all $k=1,\ldots,h$.
\end{enumerate} 
Such a sequence is called a \emph{$\delta$-sequence} and it is well known (see \cite{ab}) that there exists a polynomial $\tilde{F}(X,Y)$ with one place at infinity such that the  semigroup $\lbrace \mathrm{rank}_{\KK}\KK[X,Y]/(\tilde{F},G), G\notin (F)\rbrace$ is generated by $\underline{r}$. 

It follows from Theorem 7.1. that a polynomial curve has one place at infinity. The converse is not true in general. Abhyankar asked whether every semigroup generated by a $\delta$-sequence (hence the semigroup of a curve with one place at infinity) is the semigroup of a polynomial curve (for example, the $\delta$-sequence $(10,4,5)$ generates the semigroup $\langle 4,5\rangle$ which is the semigroup of the polynomial curve $A =\KK[x^4,x^5]$). It has been proved recently that the answer is no (\cite{fmy}). It would be nice to see which supplementary conditions a $\delta$-sequence should satisfy in order to generate the semigroup of a polynomial curve.
\end{rem}

\begin{rem}
Let $f(x)$ and $g(x)$ be as above, and let $A =\KK[f(x),g(x)]$. Let also $F(X,Y)$ be the $x$-resultant of $X-f(x)$ and $Y-g(x)$. Let $r_0=n,r_1=m,r_2,\ldots,r_h$ be the generators of $\Gamma(F)$ calculated as above. Let $1\leq k\leq h$ and let $G_k(X,Y)=\mathrm{App}(F,d_k)$. We have $\mathrm d(G_k(f(x),g(x))=r_k$, but $G_k$ is not the unique polynomial with this condition (for example, $\mathrm d((G_k+\lambda)(f(x),g(x)))=r_k$ for all $\lambda\not= 0$). Hence it is natural to ask the following: is there a polynomial $G(X,Y)$ (of degree $<n$ in $Y$) such that $G$ is parametrized by polynomials in $x$ such that $\mathrm d(G(f(x),g(x))=r_k$? Such a polynomial, if it exists,  should be of degree $\frac{n}{d_k}$ and should have the contact with $F$ at a characteristic exponent of $F$ (see \cite{ab} for the definition of the characteristic exponents of a curve with one place at infinity and the notion of contact). Hence the existence of such a polynomial implies that a polynomial curve can be approximated by polynomial curves.
\end{rem}

Let the notations be as above, in particular $F(X,Y)= Y^n+c_1(X)Y^{n-1}+\dots+c_n(X)$ is the $x$-resultant  of $(X-f(x),Y-g(x))$. Let $G_1=Y,G_2,\ldots,G_h$ be the set of approximate roots of $F(X,Y)$ constructed algorithmically as above. In particular $r_0=n,r_1=m,r_2=\mathrm{int}(F,G_2),\ldots,r_h=\mathrm{int}(F,G_h)$ generate $\mathrm d(A)$. For all $k=2,\ldots,h$, let $g_k(x)=G_k(f(x),g(x))$ and let $\mathrm M(g_k)=b_{r_k}x^{r_k}$. We have
$A =\KK[f(x),g(x),g_2(x),\ldots,g_h(x)]$. Furthermore, the map 
$$
D:\KK[u]\longrightarrow B= \KK[u][h_f,h_g,h_{g_2},\ldots,h_{g_h}]
$$
introduced in Section 6.  gives us a deformation of the polynomial curve $B\mid_{u=1}=A$ % =\KK[f(x),g(x),g_2(x),\ldots,g_h(x)]$
into $B\mid_{u=0}=\KK[t^n,t^m,t^{r_2},\ldots,t^{r_h}]$. Note that, since $\langle n,m,r_2,\ldots,r_h \rangle$ is free with respect to the given arrangement, then it is a complete intersection. For all $k\in\{1,\ldots,h\}$, write $e_kr_k=\sum_{i=0}^{k-1}\theta^k_ir_i$ with $0\leq \theta^k_i<e_i$ for every $i\in\{1,\ldots,k-1\}$. With the notations above, if $T$ is the ideal of $\KK[X_0,X_1,\ldots,X_h]$ generated by
$$
\lbrace X_1^{e_2 }-X_0^{m\over d_2}, X_2^{e_2}-X_0^{\theta^2_0}X_1^{\theta^2_1},\ldots, X_{h}^{e_h}-X_0^{\theta^{h}_0}X_1^{\theta^{h}_1}\ldots X_{h-1}^{\theta^{h}_{h-1}}\rbrace,
$$
then 
$$
\KK[x^n,x^m,x^{r_2},\ldots,x^{r_h}]\simeq {\KK[X_0,X_1,\ldots,X_h]}/{T}.
$$
Set $G_{h+1}= F$ and for all $k\geq 1$, let
$$
G_{k+1}=G_k^{e_k}-X^{\theta^k_0}\prod_{i=1}^{k-1}G_i^{\theta^k_i}+\sum_{\underline{\alpha}^k}c^k_{\underline{\alpha}^k}X^{\alpha^k_0}.G_1^{\alpha^k_1}\cdots  G_{k}^{\alpha^k_{k}},
$$
where the following conditions hold:
\begin{enumerate}
\item for all $i\in\{1,\ldots,k-1\}$, $0\leq \theta^k_i<e_i$;
\item for all $\underline{\alpha}^k$, if $c^k_{\underline{\alpha}^k}\not= 0$, then for all $i\in\{1,\ldots,k\}$, $0\leq \alpha^k_i<e_i$,
\item for all $\underline{\alpha}^k$, if $c^k_{\underline{\alpha}^k}\not= 0$, then $\alpha^k_0n+\sum_{i=1}^{k}\alpha^k_ir_i=D^k_i<e_kr_k=\theta^k_0r_0+\sum_{i=1}^{k-1}\theta^k_ir_i$.
\end{enumerate}
It follows from Section \ref{deformation-global} that if $I$ (respectively $J$) is the ideal generated by $(X_k^{e_k}-X_0^{\theta^k_0}\prod_{i=1}^{k-1}X_i^{\theta^k_i}+\sum_{\underline{\alpha}^k}c^k_{\underline{\alpha}^k}X_0^{\alpha^k_0}.X_1^{\alpha^k_1}\cdots  X_{k}^{\alpha^k_{k}})_{1\leq k\leq h}$ (respectively $(X_k^{e_k}-X_0^{\theta^k_0}\prod_{i=1}^{k-1}X_i^{\theta^k_i}+\sum_{\underline{\alpha}^k}c^k_{\underline{\alpha}^k}u^{e_kr_k-D^k_i}X_0^{\alpha^k_0}.X_1^{\alpha^k_1}\cdots  X_{k}^{\alpha^k_{k}})_{1\leq k\leq h}$) in $\KK[X_0,\ldots,X_h]$ (respectively $\KK[u][X_0,\ldots,X_h]$), then 
$$
A =\KK[x^n,g(x),g_2(x),\ldots,g_h(x)]\simeq {\KK[X_0,X_1,\ldots,X_h]}/{I}
$$
and 
$$
 \KK[u][x^n,h_{g(x)},h_{g_2(x)},\ldots,h_{g_h(x)}]\simeq {\KK[u][X_0,X_1,\ldots,X_h]}/{J}.
$$
Furthermore, ${\KK[u][X_0,X_1,\ldots,X_h]}/{J}$ is a flat $\KK[u]$-module. This gives us a family of space curves parametrized by $u$ which is a deformation from ${\KK[X_0,X_1,\ldots,X_h]}/{I}$ to 
the  toric variety ${\KK[X_0,X_1,\ldots,X_h]}/{T}$. The later being a complete intersection, we get the following result.

\begin{thm}
Every polynomial curve $X=f(x), Y=g(x)$ of $\KK^2$ has a deformation into a monomial complete intersection curve of $\KK^{h+1}$ for some positive integer $h$.
\end{thm}

\begin{ex} 
Let $f(x)=x^6+x^3, g(X)=x^4$. The minimal polynomial of $(f(x),g(x))$ is given by:
$$
F(X,Y)=Y^6-2X^2Y^3-4XY^3-Y^3+X^4.
$$
Let $r_0=6=d_1,r_1=4$ and $G_1=Y$. We have $d_2=\gcd (6,4)=2$, and $G_2=\mathrm{App}(F,2)=Y^3-X^2-2X-\frac{1}{2}$. Since  $g_2(x)=G_2(f(x),g(x))=-2x^9-3x^6-2x^3-\frac{1}{2}$, then 
$r_2=9$ and $d_3=1$, hence $\Gamma(F)=\mathrm d(A)=\langle 6,4,9 \rangle$ and $\lbrace f(x),g(x),-g_2(x)=\rbrace$ is a basis of $A $.
Consequently, $h_A =\KK[u,x^6+u^3x^3,x^4,2x^9+3u^3x^6+2u^6x^3+\frac{1}{2}u^9]$. Note that, with the notations above, $e_1=3, e_2=2$, hence $\KK[x^6,x^4,2x^9]\simeq {\KK[X_0,X_1,X_2]}/{(X_1^3-X_0^2, X_2^2-4X_0^3)}={\KK[X_0,X_1,X_2]}/T$,
$\KK[x^6+x^3,x^4,2x^9-3x^6-2x^3-\frac{1}{2}]\simeq{\KK[X_0,X_1,X_2]}/{(X_1^3-X_0^2-2X_0-\frac{1}{2}, X_2^2-4X_0^3-5X_0^2-2X_0-\frac{1}{4})}$, and
$$
\KK[u]\longrightarrow {\KK[u][X_0,X_1,X_2]}/{(X_1^3-X_0^2-2u^6X_0-\frac{1}{2}u^9, X_2^2-4X_0^3-5u^6X_0^2-2u^{12}X_0-\frac{1}{4}u^{18})}
$$
gives us a deformation from $A $ to ${\KK[X_0,X_1,X_2]}/T$.

The computation of the approximate roots and of $\Gamma(F)$ can be performed with the algorithm presented in \cite{ag}.
\begin{verbatim}
gap> f:=y^6-2*x^2*y^3-4*x*y^3-y^3+x^4;;
gap> SemigroupOfValuesOfPlaneCurveWithSinglePlaceAtInfinity(f);
<Numerical semigroup with 3 generators>
gap> MinimalGeneratingSystem(last);
[ 4, 6, 9 ]
gap> SemigroupOfValuesOfPlaneCurveWithSinglePlaceAtInfinity(f);
[ [ 6, 4, 9 ], [ y, y^3-x^2-2*x-1/2 ] ]
\end{verbatim}
\end{ex}

\begin{ex}
Let $f(x)=x^6+x, g(x)=x^4$. The minimal polynomial of $(f(x),g(x))$ is given by:
$$
F(X,Y)=Y^6-2X^2Y^3-4XY^2-Y+X^4.
$$
Let $r_0=6=d_1,r_1=4$ and $G_1=Y$. We have $d_2=\gcd (6,4)=2$, and $G_2=\mathrm{App}(F,2)=Y^3-X^2$. Since  $g_2(x)=G_2(f(x),g(x))=-2x^7-x^2$, then 
$r_2=7$ and $d_3=1$, hence $\Gamma(F)=\mathrm d(A)=\langle 6,4,7 \rangle$ and $\lbrace f(x),g(x),-g_2(x)\rbrace$ is a basis of $A $.
Consequently, $h_A =\KK[u,x^6+u^5x,x^4,2x^7+u^5x^2]$. Note that, with the notations above, $e_1=3, e_2=2$, hence $\KK[x^6,x^4,x^7]\simeq {\KK[X_0,X_1,X_2]}/{(X_1^3-X_0^2, X_2^2-X_0X_1^2)}={\KK[X_0,X_1,X_2]}/T$, and
$$
\KK[u]\longrightarrow {\KK[u][X_0,X_1,X_2]}/{(X_1^3-X_0^2, X_2^2-4X_0X_1^2-u^{10}X_1)}
$$
gives us a deformation from $A $ to ${\KK[X_0,X_1,X_2]}/T$ (we can also change $X_2$ with $\frac{1}2X_2$, and then we get 
$(X_1^3-X_0^2, X_2^2-X_0X_1^2-\dfrac{1}{4}u^{10}X_1)$ instead).
\end{ex}
      \begin{center}
         \Large
    Acknowledgements
     \end{center}

The authors would also like to thank Lance Bryant for providing plenty of examples for testing our algorithms, as well as suggesting us to implement new functionalities, that are now given as an extra argument of the function \texttt{SemigroupOfValuesOfCurve\_Local}.

A communication with the contents of this manuscript was presented by the third author at MEGA'2015. We would like to thank the anonymous referees assigned by the organization of this meeting for their comments and suggestions. All of them were taken into account in the current version of this manuscript.

\end{document}